\documentclass[reqno]{amsart}
\usepackage{geometry}
 \usepackage{float}
\usepackage{graphicx}
\usepackage{graphics}
\usepackage{subfigure}
\usepackage{mathrsfs}
\usepackage[colorlinks,linkcolor=black,anchorcolor=black,citecolor=black,hyperindex=black,CJKbookmarks=black]{hyperref}

\newtheorem{theorem}{Theorem}[section]
\newtheorem{lemma}[theorem]{Lemma}
\newtheorem{proposition}[theorem]{Proposition}
\newtheorem{definition}[theorem]{Definition}

\newtheorem{remark}[theorem]{Remark}
\newtheorem{corollary}[theorem]{Corollary}

\numberwithin{equation}{section}

\begin{document}

\title[Approximation orders of real numbers by beta-expansions]{Approximation orders of real numbers by $\beta$-expansions}

\author {Lulu Fang, Min Wu and Bing Li$^{*}$}
\address{School of Mathematics, South China University of Technology, Guangzhou 510640, P.R. China}
\email{f.lulu@mail.scut.edu.cn, wumin@scut.edu.cn and scbingli@scut.edu.cn}


\thanks {* Corresponding author}
\subjclass[2010]{Primary 11K55, 28A80; Secondary 37B10}
\keywords{ Approximation order, $\beta$-expansions, Hausdorff dimension}

\begin{abstract}
We prove that almost all real numbers (with respect to Lebesgue measure) are approximated by the convergents of their $\beta$-expansions with the exponential order $\beta^{-n}$. Moreover, the Hausdorff dimensions of sets of the real numbers which are approximated by all other orders, are determined. These results are also applied to investigate the orbits of real numbers under $\beta$-transformation, the shrinking target type problem, the Diophantine approximation and the run-length function of $\beta$-expansions.

\end{abstract}

\maketitle

\section{Introduction}

Let $\beta > 1$ be a real number and $T_\beta: [0,1] \longrightarrow [0,1]$ be the \emph{$\beta$-transformation} defined as
\begin{equation*}\label{T-beta}
T_\beta(x)= \beta x - \lfloor\beta x\rfloor,
\end{equation*}
where $\lfloor x\rfloor$ denotes the greatest integer not exceeding $x$. Then every $x \in [0,1]$ has an \emph{$\beta$-expansion}, namely
\begin{equation}\label{beta expansion}
x = \frac{\varepsilon_1(x, \beta)}{\beta} + \frac{\varepsilon_2(x, \beta)}{\beta^2} + \cdots + \frac{\varepsilon_n(x, \beta)+T^n_\beta}{\beta^n} = \sum_{n=1}^\infty \frac{\varepsilon_n(x, \beta)}{\beta^n},
\end{equation}
where $\varepsilon_1(x, \beta) = \lfloor\beta x\rfloor$ and $\varepsilon_{n+1}(x, \beta)= \varepsilon_1(T_\beta^nx, \beta)$ are called the \emph{digits} of the $\beta$-expansion of $x$~($n \in \mathbb{N}$). Sometimes we write the sequence $ \varepsilon(x,\beta)=(\varepsilon_1(x, \beta),\varepsilon_2(x, \beta),\cdots, \varepsilon_n(x, \beta),\cdots)$ as the $\beta$-expansion of $x$.
If there exists some $n_0 \in \mathbb{N}$ such that $\varepsilon_n(x, \beta) = 0$ for all $n \geq n_0$, we say that the $\beta$-expansion of $x$ is \emph{finite}. Otherwise, it is said to be \emph{infinite}. Such an expansion was first introduced by  R\'{e}nyi \cite{lesRen57} who proved that there exists a unique $T_\beta$-invariant measure $\mu_\beta$ equivalent to the Lebesgue measure when $\beta$ is not an integer;
 while it had been well known that the Lebesgue measure is $T_\beta$-invariant when $\beta$ is an integer.
Fuethermore, Gel'fond \cite{lesGel59} and Parry \cite{lesPar60} independently found the density formula for this invariant measure with respect to (w.r.t.) the Lebesgue measure.
Philipp \cite{lesPhi67} showed that the dynamical system $([0,1], \mathcal{B}, T_{\beta}, \mu_\beta)$ is an exponentially mixing measure-preserving system, where $\mathcal{B}$ is the Borel $\sigma$-algebra on $[0,1]$. Later, Hofbauer \cite{lesHof78} proved that $\mu_\beta$ is the unique measure of maximal entropy for $T_{\beta}$. Aaronson and Nakada \cite{lesAN05} obtained the $\beta$-transformation $T_\beta$ is exponential $\varphi$-mixing. And they also showed that $T_\beta$ is exponential $\psi$-mixing if and only if $\inf_{n \geq 1} T_\beta^n 1 >0$ (see Bradley \cite{lesBra05} for the definitions of $\varphi$-mixing and $\psi$-mixing). Some arithmetic, metric and fractal properties of $\beta$-expansions were studied extensively in the literature, such as \cite{lesAB07, lesBla89, lesF.W12, lesF.S92, lesKL15, lesL.W08, lesLW16, lesPS05, lesSch97, lesSch80, lesTho12} and the references therein.

For any real number $x \in [0,1]$, we denote the partial sums of the form (\ref{beta expansion}) by
\[
\omega_n(x, \beta)= \frac{\varepsilon_1(x, \beta)}{\beta} + \frac{\varepsilon_2(x, \beta)}{\beta^2} + \cdots + \frac{\varepsilon_n(x, \beta)}{\beta^n}
\]
and call them the \emph{convergents} of the $\beta$-expansion of $x$ ($n \in \mathbb{N}$). In the following, we write $\varepsilon_n(x)$ and $\omega_n(x)$ instead of $\varepsilon_n(x, \beta)$ and $\omega_n(x, \beta)$ respectively if there is no confusion. Since $T_\beta^nx \in [0,1)$, by the representation (\ref{beta expansion}), we have that the sequence $\{\omega_n(x): n \geq 1\}$ converges to $x$ as $n$ tends to infinity for any $x \in [0,1)$. A nature question is how fast $\omega_n(x)$ converges to $x$. The following theorem gives it a quantitative answer. Let $\lambda$ denote the Lebesgue measure on $[0,1]$.

\begin{theorem}\label{convergents}
Let $\beta >1$ be a real number. Then for $\lambda$-almost all $x \in [0,1)$,
\[
\lim_{n \to \infty} \frac{1}{n}\log_\beta (x -\omega_n(x)) = -1.
\]
\end{theorem}


Roughly speaking, Theorem \ref{convergents} means that $x -\omega_n(x) \approx \beta^{-n}$ for $\lambda$-almost all $x \in [0,1)$. That is to say, $x$ is approximated by its convergents $\omega_n(x)$
with exponential order $\beta^{-n}$ for $\lambda$-almost all $x \in [0,1)$.
A further question is whether there exist some points with other approximation order than $\beta^{-n}$. If yes, how large is the set of such points. More precisely, we would like to know how many real numbers can be approximated with other orders $\beta^{-\phi(n)}$, where $\phi$ is a positive function defined on $\mathbb{N}$. In other words, we are interested in the following set
\begin{equation}\label{exceptional set}
\left\{x \in [0,1): \lim_{n \to \infty} \frac{1}{\phi(n)}\log_\beta (x -\omega_n(x)) = -1 \right\}.
\end{equation}
The problem on the approximation orders is a longstanding topic in mathematics, for example, the approximation of functions or the numbers. The approximation problems on the representations of real numbers have been widely investigated, see \cite{lesF.L.W.W, lesKS07, lesPW99} for continued fractions, see \cite{lesF.W03, lesF.W04} for Oppenheim expansions, see \cite{lesBI09, lesDK96} for L\"{u}roth expansions.

Replacing the limit of the quantity $ \frac{1}{\phi(n)}\log_\beta (x -\omega_n(x))$ in (\ref{exceptional set}) with limsup, we obtain

\begin{proposition}\label{sup}
Let $\beta >1$ be a real number and $\phi$ be a positive function defined on $\mathbb{N}$. Define
\[
A_\phi:= \left\{x \in [0,1): \limsup_{n \to \infty} \frac{1}{\phi(n)}\log_\beta (x -\omega_n(x)) = -1 \right\}.
\]
Then\\
(i) If $\liminf\limits_{n \to \infty} \phi(n)/n > 1$, then $A_\phi$ is countable at most.\\
(ii) If $\limsup\limits_{n \to \infty} \phi(n)/n < 1$, then $A_\phi$ is empty.
\end{proposition}


We use the notation $\dim_{\rm H}$ to denote the Hausdorff dimension (see Falconer \cite{lesFal90}) and let $1/\infty := 0$ with the convention.
Replacing the limit with liminf in the set (\ref{exceptional set}), we have the following.

\begin{theorem}\label{inf dimension}
Let $\beta >1$ be a real number and $\phi$ be a positive function defined on $\mathbb{N}$.
Define  $\eta:= \liminf\limits_{n \to \infty} \phi(n)/n$ and
\begin{equation*}
B_\phi: = \left\{x \in [0,1): \liminf_{n \to \infty} \frac{1}{\phi(n)}\log_\beta (x -\omega_n(x)) = -1 \right\}.
\end{equation*}
Then\\
(i) If $0 \leq \eta < 1$, then $B_\phi$ is empty.\\
(ii)If additionally assume $\phi$ is nondecreasing and $\eta \geq 1$, then $\dim_{\rm H} B_\phi =1/\eta$.
\end{theorem}


Taking $\phi(n) = \alpha n$, Theorem \ref{inf dimension} gives the Hausdorff dimensions of the following level sets, which shows that these level sets have a rich multifractal structure.

\begin{corollary}\label{multifractal structure}
Let $\beta >1$ be a real number. Then the set
\[
\left\{x \in [0,1): \liminf_{n \to \infty} \frac{1}{n}\log_\beta (x -\omega_n(x)) = -\alpha \right\}
\]
has Hausdorff dimension $\alpha^{-1}$ for any $\alpha \geq 1$; otherwise it is empty.
\end{corollary}

Application of Corollary \ref{multifractal structure} implies the set of points such that Theorem \ref{convergents} does not hold, has full Hausdorff dimension.

\begin{corollary}\label{exceptional}
Let $\beta >1$ be a real number. Then
\[
\dim_{\rm H} \left\{x \in [0,1): \lim_{n \to \infty} \frac{1}{n}\log_\beta (x -\omega_n(x)) \neq -1 \right\} = 1.
\]
\end{corollary}

\section{Preliminaries}

\subsection{Basic definitions and properties for $\beta$-expansions}
\begin{definition}
An $n$-block $(\varepsilon_1 ,\varepsilon_2, \cdots, \varepsilon_n)$ is said to be admissible for $\beta$-expansions if there exists $x \in [0,1)$ such that $\varepsilon_i(x) = \varepsilon_i$ for all $1 \leq i \leq n$. An infinite sequence $(\varepsilon_1 ,\varepsilon_2, \cdots, \varepsilon_n, \cdots)$ is admissible if $(\varepsilon_1 ,\varepsilon_2, \cdots, \varepsilon_n)$ is admissible for all $n \geq 1$.
\end{definition}

We denote by $\Sigma_\beta^n$ the collection of all admissible sequences of length $n$ $(n \in \mathbb{N})$ and by $\Sigma_\beta$ that of all infinite admissible sequences. The following result of R\'{e}nyi \cite{lesRen57} implies that the dynamical system ([0,1), $T_\beta$) admits $\log \beta$ as its topological entropy.

\begin{proposition}[\cite{lesRen57}]\label{Renyi}
Let $\beta >1$ be a real number. For any $n \geq 1$,
\begin{equation*}
\beta^n \leq \sharp \Sigma_\beta^n \leq \beta^{n+1}/(\beta -1),
\end{equation*}
where $\sharp$ denotes the number of elements of a finite set.
\end{proposition}

\begin{definition}
Let $(\varepsilon_1 ,\varepsilon_2, \cdots, \varepsilon_n)\in \Sigma_\beta^n$. We define
\[
I(\varepsilon_1 ,\varepsilon_2, \cdots, \varepsilon_n) = \{x \in [0,1):\varepsilon_i(x)=\varepsilon_i \ \text{for all} \ 1 \leq i \leq n\}
\]
and call it the $n$-th cylinder of $\beta$-expansion. Furthermore, if $T_\beta^n I(\varepsilon_1 ,\varepsilon_2, \cdots, \varepsilon_n) = [0,1)$, we say $I(\varepsilon_1 ,\varepsilon_2, \cdots, \varepsilon_n)$ is full.
\end{definition}

\begin{remark}
The $n$-th order cylinder $I(\varepsilon_1 ,\varepsilon_2, \cdots, \varepsilon_n)$ is a left-closed and right-open interval with left endpoint
\begin{equation*}
\frac{\varepsilon_1}{\beta} + \frac{\varepsilon_2}{\beta^2} + \cdots + \frac{\varepsilon_n}{\beta^n}.
\end{equation*}
Moreover, the length of $I(\varepsilon_1 ,\varepsilon_2, \cdots, \varepsilon_n)$ satisfies
$|I(\varepsilon_1 ,\varepsilon_2, \cdots, \varepsilon_n)| \leq 1/\beta^n$. We stress that there is no nontrivial lower bound for the length of a $n$-th cylinder, which can be much smaller that $\beta^{-n}$.
However, it is clear that $I(\varepsilon_1 ,\varepsilon_2, \cdots, \varepsilon_n)$ is full if and only if $|I(\varepsilon_1 ,\varepsilon_2, \cdots, \varepsilon_n)| = 1/\beta^n$. The properties of full cylinders were investigated by Fan and Wang \cite{lesF.W12}. In \cite{lesB.W14}, the authors gave a full characterization of full cylinders and studied the distribution of full cylinders in the unit interval.
\end{remark}

Denote
\begin{align*}
\mathcal{A} =
\begin{cases}
\left\{0, 1, \cdots, \beta -1\right\}, & \text {if $\beta$ is an integer}; \\
\left\{0, 1, \cdots, [\beta]\right\}, & \text {if $\beta$ is not an integer}.
\end{cases}
\end{align*}
Let $(\mathcal{A}^\mathbb{N}, \sigma)$ be the symbolic dynamics with the shift transformation $\sigma$ on $\mathcal{A}^\mathbb{N}$. The finite word $\varepsilon^n$ and the infinite sequence $\varepsilon^\infty$ means $\underbrace{\varepsilon \varepsilon  \cdots\varepsilon}_{n}$ and $\varepsilon \varepsilon  \cdots\varepsilon \cdots$ respectively for a finite word $\varepsilon \in \mathcal{A}^\mathbb{N}$ $(n \in \mathbb{N})$. When $\beta$ is an integer, $\Sigma_\beta$ is simply $\mathcal{A}^\mathbb{N}$ (or more precisely $\mathcal{A}^\mathbb{N} = \mathcal{S}_\beta$ defined below); when $\beta$ is not an integer, $\Sigma_\beta$ was characterized by Parry \cite{lesPar60} using the infinite $\beta$-expansion of the number 1, denoted by $\varepsilon^*(1, \beta)$, which can be obtained in the following: if the $\beta$-expansion of the number 1 is finite, i.e., $\varepsilon(1, \beta) =(\varepsilon_1(1),\varepsilon_2(1),\cdots, \varepsilon_n(1),0^\infty)$ with $\varepsilon_n(1) \neq 0$ for some $n \geq 1$, we define by $\varepsilon^*(1, \beta)$ the  infinite $\beta$-expansion of the number 1 as $(\varepsilon_1(1),\varepsilon_2(1),\cdots, \varepsilon_{n-1}(1),\varepsilon_n(1)-1)^\infty$; if the
$\beta$-expansion of the number 1 is infinite, we keep it and still write $\varepsilon^*(1, \beta)$ instead of $\varepsilon(1, \beta)$ in this case.
To state the following proposition, we give two notations $\prec$ and $\preceq$, the lexicographical orders on $\mathcal{A}^\mathbb{N}$. That is, let $\varepsilon = \varepsilon_1 \varepsilon_2 \cdots \varepsilon_n \cdots$ and $\varepsilon^{\prime}=\varepsilon_1^{\prime} \varepsilon_2^{\prime} \cdots \varepsilon_n^{\prime} \cdots$ both belong to $\mathcal{A}^\mathbb{N}$, then $\varepsilon \prec \varepsilon^{\prime}$ means that there exists $k \geq 1$ such that $\varepsilon_i = \varepsilon_i^{\prime}$ for all $1 \leq i < k$ and $\varepsilon_k < \varepsilon_k^{\prime}$, and $\varepsilon \preceq \varepsilon^{\prime}$ means that $\varepsilon \prec \varepsilon^{\prime}$ or $\varepsilon = \varepsilon^{\prime}$.

\begin{proposition}(\cite[Theorem 3]{lesPar60})\label{parry pr}
Let $\beta >1$ be a real number and $\varepsilon^*(1, \beta)$ be the infinite $\beta$-expansion of the number 1. \\
(i) $\omega \in \Sigma_\beta$ if and only if
\[
\sigma^n(\omega) \prec \varepsilon^*(1, \beta)\ \  \text{for all} \ \ n \geq 0.
\]
(ii) The function $\beta \mapsto \varepsilon^*(1, \beta)$ is increasing w.r.t. $\beta$. Therefore, if $1 < \beta_1 < \beta_2$, then
\[
\Sigma_{\beta_1} \subset \Sigma_{\beta_2},\ \ \  \Sigma^n_{\beta_1} \subset \Sigma^n_{\beta_2}\ \ \text{for all}\ \ n \geq 1.
\]
\end{proposition}

Let $\mathcal{S}_\beta$ be the closure of the set $\Sigma_\beta$. It is clear to see $\mathcal{S}_\beta= \mathcal{A}^\mathbb{N}$ when $\beta$ is an integer and otherwise, $(\mathcal{S}_\beta, \sigma|_{\mathcal{S}_\beta})$ is a subshift of $(\mathcal{A}^\mathbb{N}, \sigma)$, where $\sigma|_{\mathcal{S}_\beta}$ is the restriction of $\sigma$ to $\mathcal{S}_\beta$. Proposition \ref{parry pr} implies the following characterization of $\mathcal{S}_\beta$.

\begin{corollary}[\cite{lesPar60}]
Let $\beta >1$ be a real number and $\varepsilon^*(1, \beta)$ be the infinite $\beta$-expansion of the number 1. Then
\[
\mathcal{S}_\beta = \left\{\omega \in \mathcal{A}^\mathbb{N}: \sigma^n(\omega) \preceq \varepsilon^*(1, \beta)\ \ \text{for all} \ \ n \geq 0 \right\}.
\]
\end{corollary}

In 1989, Blanchard \cite{lesBla89} outlined a classification for all numbers $\beta >1$ according to the topological properties of $\mathcal{S}_\beta$. Later, the Lebesgue measures and Hausdorff dimensions of all classes were calculated by Schmeling \cite{lesSch97}. Denote by $\ell_n(1,\beta)$ the length of the longest strings of zeroes just after the $n$-th digit in the $\beta$-expansion of the number 1, namely,
\[
\ell_n(1,\beta) = \sup\left\{k \geq 0: \varepsilon^*_{n +1}(1)=\cdots=\varepsilon^*_{n +k}(1) =0\right\}.
\]
Recently, Li and Wu \cite{lesL.W08} provided another classification of $\beta >1$ by the growth of $\ell_n(1,\beta)$ as follows:
\[
A_0 = \left\{\beta>1: \limsup_{n \to \infty}\ell_n(1,\beta)< + \infty,~\text{i.e.},~\left\{\ell_n(1,\beta)\right\}_{n \geq 1}\ \text{is bounded}\right\}
\]
and $A_1 = (1,+\infty)\backslash A_0$. Then $A_0$ is dense in $(1,+\infty)$ (see Parry \cite{lesPar60}).
It is worth pointing out that all $\beta$'s such that $\mathcal{S}_\beta$ is a subshift of finite type are contained in $A_0$ and $\beta \in A_0$ if and only if $\mathcal{S}_\beta$ satisfies the specification property. Buzzi \cite{lesBuz97} proved that the set of $\beta>1$ such that the $\beta$-transformation $T_\beta$ has the specification property is of zero Lebesgue measure. Furthermore, Schmeling \cite{lesSch97} proved that $A_0$ has full Hausdorff dimension and Li et al.~\cite{lesL.P.W.W14} showed that $A_1$ is of full Lebesgue measure.

The following lemma from \cite{lesL.W08} gives a way to get full cylinders.

\begin{lemma}(\cite{lesL.W08})\label{full}
Let $\beta >1$ be a real number and $(\varepsilon_1, \varepsilon_2, \cdots, \varepsilon_n) \in \Sigma^n_\beta$. Denote $M_n(\beta)= \max_{1 \leq k \leq n}\{\ell_k(1,\beta)\}$, then for any $m > M_n(\beta)$, the cylinder
\[
I(\varepsilon_1, \varepsilon_2, \cdots, \varepsilon_n, \underbrace{0,\cdots,0}_{m})
\]
is a full cylinder and its length equals $\beta^{-(n+m)}$.
\end{lemma}

The following result characterizes the sizes of cylinders by the classification in $A_0$.

\begin{proposition}(\cite{lesL.W08})\label{A0}
$\beta \in A_0$ if and only if there exists a positive constant $C_0$ such that for all $x \in [0,1)$ and $n \geq 1$,
\[
C_0 \beta^{-n} \leq |I(\varepsilon_1(x), \varepsilon_2(x), \cdots, \varepsilon_n(x))| \leq \beta^{-n}.
\]
\end{proposition}

\subsection{Approximation method for the $\beta$-shift}\label{Approximation method}
Define a projection function $\pi_\beta: \mathcal{S}_\beta \longrightarrow [0,1]$ as following
\[
\pi_\beta(\omega) = \sum_{i=1}^\infty \frac{\omega_i}{\beta^i}
\]
for any $\omega = (\omega_1,\omega_2,\cdots,\omega_n,\cdots) \in \mathcal{S}_\beta$.
Then $\pi_\beta$ is one-to-one except at the countable many points for which their $\beta$-expansions are finite and the restriction of $\pi_\beta$ to which is two-to-one. Let $\beta > \beta^\prime >1$. Since $\Sigma_{\beta^\prime} \subset \Sigma_\beta$, we know that
\[
H_\beta^{\beta^\prime}:= \pi_\beta(\Sigma_{\beta^\prime}) = \left\{ \sum_{i=1}^\infty \frac{\omega^\prime_i}{\beta^i}:(\omega^\prime_1,\cdots,\omega^\prime_n,\cdots) \in \Sigma_{\beta^\prime} \right\}
\]
is a Cantor set of $\pi_\beta(\Sigma_\beta)=[0,1)$. Define the function $h: H_\beta^{\beta^\prime} \longrightarrow [0,1)$ as
\[
h(x) = \pi_{\beta^\prime}(\varepsilon(x,\beta))
\]
for any $x \in H_\beta^{\beta^\prime}$.

\begin{proposition}(\cite{lesBL14})\label{h}
Let $\beta>\beta^\prime>1$.\\
(i) For any $x \in H_\beta^{\beta^\prime}$, we have $\varepsilon(h(x),\beta^\prime)= \varepsilon(x,\beta)$.\\
(ii) The function $h$ is bijective and strictly increasing on $H_\beta^{\beta^\prime}$.\\
(iii) The function $h$ is continuous on $H_\beta^{\beta^\prime}$.\\
(iv) If additionally assume that $\beta^\prime \in A_0$ with $M=\sup\{\ell_n(1,\beta^\prime): n \geq 1\}$, then $h$ is H\"{o}lder continuous on $H_\beta^{\beta^\prime}$. More precisely,
\[
|h(x)-h(y)| \leq \beta^{\prime{M+2}}|x-y|^{\frac{\log\beta^\prime}{\log\beta}}
\]
for any $x,y \in H_\beta^{\beta^\prime}$.
\end{proposition}

The function $h$ induces a method to provide a lower bound of the Hausdorff dimension of a given set $E \subset [0,1)$. Firstly, consider a subset $E\cap H_\beta^{\beta^\prime}$ of $E$ and use the H\"{o}lder function $h$ in Proposition \ref{h} to transfer it to $h(E\cap H_\beta^{\beta^\prime})$, whose Hausdorff dimension may be easier to be obtained by choosing $\beta^\prime \in A_0$ or $\beta^\prime$ satisfying that $\mathcal{S}_{\beta^\prime}$ is subshift of finite type. Secondly, give a lower bound of the Hausdorff dimension of $h(E\cap H_\beta^{\beta^\prime})$ and then by the H\"{o}lder exponent of $h$ in Proposition \ref{h} (iv) have a lower bound of the Hausdorff dimension of $E\cap H_\beta^{\beta^\prime}$, also that of $E$. That is,
\[
\dim_H E \geq \dim_H (E\cap H_\beta^{\beta^\prime}) \geq \frac{\log \beta^\prime}{\log \beta} \dim_H h(E\cap H_\beta^{\beta^\prime}).
\]
Finally, let $\beta^\prime$ approximates to $\beta$. In Section 4, we will apply this approximation method to prove our desired results.

\section{Metric results}

For any real number $ x \in [0,1)$ and $n \geq 1$, it is clear to see that
\begin{equation}\label{jie}
\frac{1}{\beta^{n + \ell_n(x)+1}} \leq x -\omega_n(x) \leq \frac{1}{\beta^{n + \ell_n(x)}}
\end{equation}
where $\ell_n(x) = \sup\left\{k \geq 0: \varepsilon_{n +1}(x)=\cdots=\varepsilon_{n +k}(x) =0 \right\}$ is the length of the longest string of zeros just after the $n$-th digit in the $\beta$-expansion of $x$. The quantity $\ell_n(x)$ has been used by Fang, Wu and Li \cite{lesFWL16, lesFWLarXiv}. The inequalities (\ref{jie}) indicate that $\ell_n(x)$ plays an important role in the approximation theory of $\beta$-expansions.

To estimate $\ell_n(x)$, we first define $r_n(x)$ the maximal length of the strings of zero's in the block of the first $n$ digits of the $\beta$-expansion of $x \in [0,1)$. That is,
\begin{equation*}
r_n(x) = \sup\{k \geq 0: \varepsilon_{i+1}(x)= \cdots = \varepsilon_{i+k}(x) =0\ \text{for some}\  0 \leq i \leq n-k\}.
\end{equation*}
Sometimes the quantity $r_n(x)$ is called \emph{run-length function} for $\beta$-expansions.
Tong et al.~\cite{lesTYZ16} studied the order of magnitude of $r_n(x)$ and investigated the Hausdorff dimensions for the corresponding exceptional sets.

\begin{lemma}(\cite[Theorem 1.1]{lesTYZ16})\label{sup ln xia}
Let $\beta >1$ be a real number. Then for $\lambda$-almost all $x \in [0,1)$,
\begin{equation*}
\lim_{n \to \infty} \frac{r_n(x)}{\log_\beta n} = 1.
\end{equation*}
\end{lemma}

Combing this with the relation between $r_n(x)$ and $\ell_n(x)$, we have the following

\begin{proposition}\label{sup ln}
Let $\beta >1$ be a real number. Then for $\lambda$-almost all $x \in [0,1)$,
\begin{equation*}
\limsup_{n \to \infty} \frac{\ell_n(x)}{\log_{\beta} n} =1.
\end{equation*}
Moreover, for any real number $x \in [0,1)$ whose $\beta$-expansion is infinite, we have
\begin{equation*}
\liminf_{n \to \infty} \frac{\ell_n(x)}{\log_\beta n} = 0.
\end{equation*}
\end{proposition}

\begin{proof}
We first prove the result of liminf part.
Let $x \in [0,1)$ be a real number whose $\beta$-expansion is infinite. Then there exists a subsequence of digits $\{\varepsilon_{n_k}(x): k \geq 1\}$ with $\varepsilon_{n_k}(x) \neq 0$ for all $k \geq 1$. So, by the definition of $\ell_n(x)$, we have $\ell_{n_k-1}(x) = 0$ for any $k \geq 1$ and hence
\begin{equation*}
\liminf_{n \to \infty} \frac{\ell_n(x)}{\log_\beta n} \leq \liminf_{k \to \infty} \frac{\ell_{n_k-1}(x)}{\log_\beta (n_k-1) } =0.
\end{equation*}
Therefore, $\liminf\limits_{n \to \infty} \ell_n(x)/(\log_\beta n) = 0$ by the definition of $\ell_n(x)$.
Now we turn to the result for limsup part.
Let $B$ be the set such that the Lemma \ref{sup ln xia} does not hold and let $A = [0,1)\backslash B$. Then $\lambda(A)=1$. For any $x \in A$ and $n \geq 1$, we know that $r_{n+\ell_n(x)}(x) = \max_{1 \leq k \leq n}\ell_k(x)$ by the definitions of $\ell_n(x)$ and $r_n(x)$. So there exists $1 \leq k_n := k_n(x) \leq n$ such that $\ell_{k_n}(x) = \max_{1 \leq k \leq n}\ell_k(x)$. Therefore, $r_{n+\ell_n(x)}(x) = \ell_{k_n}(x)$ and hence
\[
\frac{\ell_{k_n}(x)}{\log_\beta k_n} \geq \frac{r_{n+\ell_n(x)}(x)}{\log_\beta n} \geq \frac{r_{n+\ell_n(x)}(x)}{\log_\beta (n+\ell_n(x))}.
\]
Combining this with Lemma \ref{sup ln xia}, we deduce that
\[
\limsup_{n \to \infty} \frac{\ell_n(x)}{\log_\beta n} \geq \limsup_{n \to \infty} \frac{\ell_{k_n}(x)}{\log_\beta k_n} \geq \liminf_{n \to \infty} \frac{r_{n+\ell_n(x)}(x)}{\log_\beta (n+\ell_n(x))} \geq \liminf_{n \to \infty} \frac{r_{n}(x)}{\log_\beta n} = 1.
\]
On the other hand, note that $r_{n+\ell_n(x)}(x) = \max_{1 \leq k \leq n}\ell_k(x) \geq \ell_n(x)$, so
\[
\limsup_{n \to \infty} \frac{\ell_n(x)}{\log_\beta n} \leq \limsup_{n \to \infty} \frac{r_{n+\ell_n(x)}(x)}{\log_\beta (n+\ell_n(x))} \leq \limsup_{n \to \infty} \frac{r_n(x)}{\log_\beta n} =1.
\]
Therefore, we get that $\limsup\limits_{n \to \infty} \ell_n(x)/(\log_\beta n) = 1$ for $\lambda$-almost all $x \in [0,1)$.
\end{proof}

Now we are ready to prove Theorem \ref{convergents} and Proposition \ref{sup}.

\begin{proof}[Proof of Theorem \ref{convergents}]
For any $x \in [0,1)$ and $n \geq 1$, by (\ref{jie}), we obtain that
\[
\lim_{n \to \infty} \frac{1}{n}\log_\beta (x - \omega_n(x)) =- \lim_{n \to \infty} \frac{\ell_n(x)}{n}-1.
\]
In view of Proposition \ref{sup ln}, we know that $\lim\limits_{n \to \infty} \ell_n(x)/n =0$ for $\lambda$-almost all $x \in [0,1)$. Thus,
\[
\lim_{n \to \infty} \frac{1}{n}\log_\beta (x - \omega_n(x)) =- 1
\]
for $\lambda$-almost all $x \in [0,1)$.
\end{proof}

\begin{proof}[Proof of Proposition \ref{sup}]
(i): Notice that $\phi(n) \to \infty$ as $n \to \infty$ since $\liminf\limits_{n \to \infty}\phi(n)/n >1$.
So, it follows from the inequalities (\ref{jie}) that
\begin{equation*}
\limsup_{n \to \infty} \frac{1}{\phi(n)}\log_\beta (x -\omega_n(x)) = -1\ \ \  \text{if and only if}\ \ \
\liminf_{n \to \infty} \frac{n+\ell_n(x)}{\phi(n)} = 1
\end{equation*}
for any $x \in [0,1)$.
We claim that
\begin{equation*}
\left\{x \in[0,1) : \liminf_{n \to \infty} \frac{n+\ell_n(x)}{\phi(n)} = 1 \right\} \subseteq
 \left\{x \in[0,1): \text{the $\beta$-expansion of $x$ is finite} \right\}.
\end{equation*}
In fact, suppose that $x \in [0,1)$ whose $\beta$-expansion is infinite, that is, there exists a subsequence $\{n_k\}_{k\geq 1}$ such that $\varepsilon_{n_k-1}(x) \neq 0$. Then $\ell_{n_k}(x) = 0$ by the definition of $\ell_n(x)$ and hence
\[
\liminf_{n \to \infty} \frac{n+\ell_n(x)}{\phi(n)} \leq \liminf_{k \to \infty} \frac{n_k}{\phi(n_k)} = \frac{1}{\limsup\limits_{k \to \infty}\phi(n_k)/n_k} \leq \frac{1}{\liminf\limits_{k \to \infty}\phi(n_k)/n_k}.
\]
Therefore,
\[
\liminf_{n \to \infty} \frac{n+\ell_n(x)}{\phi(n)} \leq \frac{1}{\liminf\limits_{n \to \infty}\phi(n)/n}<1.
\]
Thus we get the desired result since the set of the points  with finite $\beta$-expansions is a countable set.

(ii): Since $\limsup\limits_{n \to \infty}\phi(n)/n <1$, we deduce that
\[
\liminf_{n \to \infty} \frac{n+\ell_n(x)}{\phi(n)} \geq \liminf_{n \to \infty} \frac{n}{\phi(n)} = \frac{1}{\limsup\limits_{n \to \infty}\phi(n)/n} >1.
\]
By (\ref{jie}), we know that
\[
\frac{\log_\beta (x -\omega_n(x))}{\phi(n)} \leq - \frac{n+\ell_n(x)}{\phi(n)},
\]
which implies that
\[
\limsup_{n \to \infty} \frac{1}{\phi(n)}\log_\beta (x -\omega_n(x))< -1.
\]
\end{proof}

\section{Dimensional results}
The inequalities (\ref{jie}) say that
\begin{equation*}
\left\{x \in [0,1): \liminf_{n \to \infty} \frac{1}{\phi(n)}\log_\beta (x -\omega_n(x)) = -1 \right\} =
 \left\{x \in [0,1): \limsup_{n \to \infty} \frac{n + \ell_n(x)}{\phi(n)} = 1 \right\}
\end{equation*}
if $\phi(n) \to \infty$ as $n \to \infty$. Now it leads us to consider the Hausdorff dimension of the set
\begin{equation*}
F_\phi= \left\{x \in [0,1): \limsup_{n \to \infty} \frac{\ell_n(x)}{\phi(n)} = 1 \right\}.
\end{equation*}
Firstly, we will give an upper bound for the Hausdorff dimension of the set $F_\phi$.

\subsection{Upper bound}
Denote
\[
E_\phi =  \big\{x \in [0,1): \ell_n(x) \geq \phi(n)\ i.o.\big\},
\]
where $i.o.$ means infinitely often. It is clear that $F_\phi \subseteq E_{(1-\delta)\phi}$ for any $0< \delta <1$. So we will determine the upper bound of the Hausdorff dimension of the set $F_\phi $ by giving an upper bound of the Hausdorff dimension of $E_\phi$.

\begin{lemma}\label{E}
Let $\beta >1$ be a real number. Assume that $\phi$ is a nonnegative function defined on $\mathbb{N}$.
Then
\begin{equation*}
\dim_{\rm H} E_\phi \leq \frac{1}{1+\liminf\limits_{n \to \infty} \phi(n)/n}.
\end{equation*}
\end{lemma}

\begin{proof}
The upper bound can be obtained by considering the natural covering system.
Notice that
\[
E_\phi = \left\{x \in [0,1): \ell_n(x) \geq \phi(n)\ i.o.\right\} = \bigcap_{N=1}\bigcup_{n=N}\left\{x \in [0,1): \ell_n(x) \geq \phi(n)\right\}
\]
and
\[
\left\{x \in [0,1): \ell_n(x) \geq \phi(n)\right\}  \subseteq \bigcup_{(\varepsilon_1,\cdots,\varepsilon_n) \in \Sigma_\beta^n} I(\varepsilon_1,\cdots,\varepsilon_n,\underbrace{0,\cdots,0}_{\lfloor\phi(n)\rfloor}),
\]
so, for any $N \geq 1$, we obtain that
\[
E_\phi \subseteq \bigcup_{n=N} \bigcup_{(\varepsilon_1,\cdots,\varepsilon_n) \in \Sigma_\beta^n} I(\varepsilon_1,\cdots,\varepsilon_n,\underbrace{0,\cdots,0}_{\lfloor\phi(n)\rfloor}).
\]
Let $s = 1/(1+\liminf\limits_{n \to \infty} \phi(n)/n)$. Then $0 \leq s \leq 1$. For any $\delta >0$ and $N \geq 1$, by the definition of Hausdorff measure, we have that
\begin{align}\label{measure}
\mathcal{H}^{s+\delta}(E_\phi) &\leq \sum_{n=N}^\infty \sum_{(\varepsilon_1,\cdots,\varepsilon_n) \in \Sigma_\beta^n} |I(\varepsilon_1,\cdots,\varepsilon_n,\underbrace{0,\cdots,0}_{\lfloor\phi(n)\rfloor})|^{s+\delta}\notag \\
&\leq\sum_{n=N}^\infty\frac{\beta^{n+1}}{\beta-1}\cdot\left(\frac{1}{\beta^{n+\lfloor\phi(n)\rfloor}}\right)^{s+\delta}
= \frac{\beta}{\beta-1} \sum_{n=N}^\infty \left(\frac{1}{\beta^{n}}\right)^{t_n},
\end{align}
where $t_n := (1+ \lfloor\phi(n)\rfloor/n)(s+\delta) -1$ and the second inequality follows from Proposition \ref{Renyi} and the fact $|I(\varepsilon_1,\cdots,\varepsilon_n)| \leq 1/\beta^{n}$ for any $(\varepsilon_1,\cdots,\varepsilon_n) \in \Sigma_{\beta}^n$.

When $s=0$, i.e., $\liminf\limits_{n \to \infty} \phi(n)/n = +\infty$. So, $\liminf\limits_{n \to \infty} \lfloor\phi(n)\rfloor/n \geq \delta^{-1}$. So we know that
\[
\liminf\limits_{n \to \infty} t_n = \liminf\limits_{n \to \infty} \left(1 + \frac{\lfloor\phi(n)\rfloor}{n}\right)\delta -1 \geq (1+\delta^{-1})\delta -1 = \delta.
\]

When $0<s \leq 1$, we deduce that
\[
\liminf\limits_{n \to \infty} t_n = \liminf\limits_{n \to \infty} \left(1 + \frac{\lfloor\phi(n)\rfloor}{n}\right)(s+\delta) -1 = (s+\delta)/s -1 = \delta/s \geq \delta.
\]
In conclusion, we have that $\liminf\limits_{n \to \infty} t_n  \geq \delta$. So, there exists $N_\delta>0$ such that $t_n \geq \delta/2$ for all $n \geq N_\delta$. Applying $N=N_\delta$ to (\ref{measure}), we obtain that
\[
\mathcal{H}^{s+\delta}(E_\phi) \leq \frac{\beta}{\beta-1} \sum_{n=N_\delta}^\infty \left(\frac{1}{\beta^{n}}\right)^{t_n} \leq \frac{\beta}{\beta-1} \sum_{n=N_\delta}^\infty \frac{1}{\beta^{n\delta/2}} < +\infty.
\]
Therefore, $\dim_{\rm H} E_\phi \leq s+\delta$. By the arbitrariness of $\delta >0$, we have $\dim_{\rm H} E_\phi \leq s$.
\end{proof}

Note that $F_\phi \subset E_{(1-\delta)\phi}$ for any $0< \delta<1$, by Lemma \ref{E}, we know that
\[
\dim_{\rm H} F_\phi \leq \dim_{\rm H} E_{(1-\delta)\phi} \leq \frac{1}{1+(1-\delta)\liminf\limits_{n \to \infty} \phi(n)/n}.
\]
Let $\delta \to 0^+$, we deduce that $$\dim_{\rm H} F_\phi \leq \frac{1}{1+\liminf\limits_{n \to \infty} \phi(n)/n}. $$

\subsection{Lower bound}
In this subsection, the main aim is to determine a lower bound for the Hausdorff dimension of the set $F_\phi$. First, we give a lower bound of the Hausdorff dimension of $F_\phi$ for $\beta \in A_0$, where $A_0$ is as defined in Section 2.1. Then go on to estimate the lower bound of the Hausdorff dimension of $F_\phi$ for any $\beta >1$ using the approximation method stated in Section \ref{Approximation method}.

The lower bound of the Hausdorff dimension of $F_\phi$ is yielded by constructing a Cantor-like subset of $F_\phi$. The mass distribution principle (see \cite[Proposition 4.2]{lesFal90}) is a classical tool to give a lower bound estimation for the Hausdorff dimension of a set. In the classical form of the mass distribution principle, we need to estimate the measure of an arbitrary ball, while the following modified mass distribution principle tells us that, for $\beta$-expansions, it is sufficient to consider only the measure of cylinders.

\begin{proposition}(\cite[Proposition 1.3]{lesB.W14})\label{MMDP}
Let $E$ be a Borel measurable set in $[0,1)$ and $\nu$ be a Borel measure with $\nu(E)>0$. Assume that there exist a positive constant $c$ and an integer $n_0$ such that
\[
\nu(I_n) \leq c|I_n|^s
\]
for any $n \geq n_0$ and any $n$-th order cylinder $I_n$. Then $\dim_{\rm H} E \geq s$.
\end{proposition}

\subsubsection{The cases of bases in $A_0$}

\begin{lemma}\label{F}
Let $\beta \in A_0$.
Assume that $\phi$ is a positive and nondecreasing function defined on $\mathbb{N}$ with $\phi(n) \to \infty$ as $n \to \infty$. Then
\begin{equation}\label{lowerF}
\dim_{\rm H} F_\phi \geq \frac{1}{1+\liminf\limits_{n \to \infty} \phi(n)/n}.
\end{equation}
\end{lemma}

The following is devoted to proving this lemma in this subsection.

Let $0 \leq s = 1/(1+\liminf\limits_{n \to \infty} \phi(n)/n)\leq 1$. When $s =0$, (\ref{lowerF}) holds trivially.
Now let $0 < s \leq 1$. Denote $M= \sup_{n \geq 1}{\ell_n(1,\beta)}$ and $m = M+1$, then we have that $0 < M, m< +\infty$ since $\beta \in A_0$. Let $\{n_i\}_{i \geq 1}$ be a sequence with
\begin{equation}\label{zilie}
\lim_{i \to \infty} \frac{\phi(n_i)}{n_i} = \liminf_{n \to \infty} \frac{\phi(n)}{n}.
\end{equation}
For any $0< \delta < s$, we choose a largely sparse subsequence $\{n_{i_j}\}_{j \geq 1}$ of $\{n_i\}_{i \geq 1}$ (for simplicity, we still denote by $\{n_i\}_{i \geq 1}$) such that
\begin{align*}
&\frac{1}{\lfloor\phi(n_1)\rfloor} <\frac{\delta}{4(m+1-\log_{\beta}(\beta-1))},\ \, n_1 >1, \ \
n_i > n_{i-1} + 1+ \lfloor\phi(n_{i-1})\rfloor
\end{align*}
and
\begin{align}\label{tiaojian}
\frac{\delta}{2}&\big(n_i+ \lfloor\phi(n_i)\rfloor\big) \geq \sum_{j=1}^{i-1} \big(\lfloor\phi(n_j)\rfloor+r_j+1\big) + \big(m+1- \log_{\beta}(\beta-1)\big)\sum_{j=1}^{i-1}(k_j +1),
\end{align}
where $k_j = \left\lfloor\frac{n_{j+1} - n_j -1}{[\phi(n_j)]}\right\rfloor \geq 1$ is an integer and $r_j = n_{j+1} -n_j -1 -k_j\lfloor\phi(n_j)\rfloor$ is the remainder. We sometimes write the right-hand side formula of (\ref{tiaojian}) as $U_i$ (with the convention $U_1 :=0$).
The first condition $\lfloor\phi(n_1)\rfloor^{-1} < \frac{\delta}{4(m+1-\log_{\beta}(\beta-1))}$ assures that the sequence satisfying (\ref{tiaojian}) can be chosen. Moreover, it also guarantees that
\[
\frac{1}{\lfloor\phi(n_i)\rfloor} < \frac{1+\delta -s}{m+1-\log_{\beta}(\beta-1)}
\]
for $0< \delta < s$. That is, $W_i := \lfloor\phi(n_i)\rfloor(s-\delta-1) + m+1-\log_{\beta}(\beta-1) <0$. The following proof will be divided into three steps to make use of Proposition \ref{MMDP}.

{\bf Step 1 Construction of a Cantor-like subset.} We denote the two subsets of integers
\[
\mathcal{I}_1 =\{n_i + j: i\geq 1,1\leq j\leq \lfloor\phi(n_i)\rfloor\} \cup \{n_i + j\lfloor\phi(n_i)\rfloor-m +r: i\geq 1, 2\leq j \leq k_i, 1\leq r \leq m \}
\cup\Gamma
\]
and
\[
\mathcal{I}_2 = \{ n_i, n_i + j\lfloor\phi(n_i)\rfloor +1: i\geq 1, 1\leq j \leq k_i -1\} \backslash \{n_1\} \cup \Lambda,
\]
where
$\Gamma = \cup_{i \geq 1} \Gamma_i$, $\Lambda = \cup_{i \geq 1} \Lambda_i$ and when $r_i \leq m$, $\Gamma_i = \{ n_i + k_i\lfloor\phi(n_i)\rfloor+1, \cdots, n_{i+1}-1\}$ and $\Lambda_i$ is empty; if $r_i > m$, $\Gamma_i=
\{ n_{i+1}-m, \cdots, n_{i+1}-1\}$ and $\Lambda_i = \{n_i + k_i\lfloor\phi(n_i)\rfloor +1\}$.
Let $\mathcal{D}_n$ be the $n$-th order cylinders $I(\varepsilon_1, \cdots, \varepsilon_n)$ satisfying that $\varepsilon_k = 0$ if $k \in I_1$, $\varepsilon_k \neq 0$ if $k \in I_2$, and $\varepsilon_k \in \mathcal{A}$ if $k \notin \mathcal{I}_1 \cup \mathcal{I}_2$. Put
\[
F = \bigcap_{n=1}^\infty \bigcup_{I(\varepsilon_1, \cdots, \varepsilon_n) \in \mathcal{D}_n}I(\varepsilon_1, \cdots, \varepsilon_n).
\]
By the constructions of $\mathcal{D}_n$ and $F$, and the monotonic property of $\phi$, we claim that
\[
F \subset F_\phi = \left\{x \in [0,1): \limsup_{n \to \infty} \frac{\ell_n(x)}{\phi(n)} = 1 \right\}.
\]
In fact, for any $n \geq 1$, there exists $k \in \mathbb{N}$ such that $n_k \leq n < n_{k+1}$. For any $x \in F$, by the construction of F and the monotonic property of $\phi$, we know that
\[
\ell_n(x) = \max\{\lfloor\phi(n_k)\rfloor\ , \lfloor\phi(n_{k-1})\rfloor\} \leq \lfloor\phi(n_k)\rfloor\ \leq \phi(n)
\]
and hence that
\[
\limsup_{n \to \infty} \frac{\ell_n(x)}{\phi(n)}  \leq 1.
\]
Note that $\ell_{n_k}(x) = \lfloor\phi_{n_k}\rfloor$, so $\lim\limits_{k \to \infty} \ell_{n_k}(x)/\phi(n_k) = 1$. Therefore, $\limsup\limits_{n \to \infty} \ell_n(x)/\phi(n) =1$.

{\bf Step 2 Supporting measure.} We now distribute a probability measure $\nu$ supported on $F$. We first give the definition of $\nu$ on cylinders. Notice that $n_1 >1$, then $1 \notin \mathcal{I}_1 \cup \mathcal{I}_2$. For any $I(\varepsilon_1) \in \mathcal{D}_1$, we define $\nu(I(\varepsilon_1)) = |I(\varepsilon_1)|$ and $\nu(I(\varepsilon_1)) = 0$ if $I(\varepsilon_1) \notin \mathcal{D}_1$. Suppose that $\nu(I(\varepsilon_1, \cdots, \varepsilon_n))$ is well defined for any $I(\varepsilon_1, \cdots, \varepsilon_n)) \in \mathcal{D}_n$, we define $\nu(I(\varepsilon_1, \cdots, \varepsilon_n,\varepsilon_{n+1}))$ as follows:
\[
\nu(I(\varepsilon_1, \cdots, \varepsilon_n,\varepsilon_{n+1})) =
\begin{cases}
\nu(I(\varepsilon_1, \cdots, \varepsilon_n)) &\text{if $n+1 \in \mathcal{I}_1$}; \\
\frac{|I(\varepsilon_1, \cdots, \varepsilon_n,\varepsilon_{n+1})|}{|I(\varepsilon_1, \cdots, \varepsilon_n)\backslash I(\varepsilon_1, \cdots, \varepsilon_n,0)|} \cdot
\nu(I(\varepsilon_1, \cdots, \varepsilon_n)) &\text{if $n+1 \in \mathcal{I}_2$}; \\
\frac{|I(\varepsilon_1, \cdots, \varepsilon_n,\varepsilon_{n+1})|}{|I(\varepsilon_1, \cdots, \varepsilon_n)|} \cdot \nu(I(\varepsilon_1, \cdots, \varepsilon_n)) &\text{if $n+1 \notin \mathcal{I}_1 \cup \mathcal{I}_2$},
\end{cases}
\]
and $\nu(I(\varepsilon_1, \cdots, \varepsilon_n,\varepsilon_{n+1})) =0$ if $I(\varepsilon_1, \cdots, \varepsilon_n,\varepsilon_{n+1}) \notin \mathcal{D}_{n+1}$. Thus, the measure $\nu$ is well-defined on all cylinders since we can verify that
\[
\sum_{(\varepsilon_1,\cdots,\varepsilon_{n}) \in \Sigma^n_\beta} \nu(I(\varepsilon_1, \cdots, \varepsilon_n)) =1
\]
and
\[
\sum_{(\varepsilon_1,\cdots,\varepsilon_n, \varepsilon_{n+1}) \in \Sigma^{n+1}_\beta} \nu(I(\varepsilon_1, \cdots, \varepsilon_n,\varepsilon_{n+1})) =  \nu(I(\varepsilon_1, \cdots, \varepsilon_n)).
\]
Notice that the set of all cylinders forms a semi-algebra, by Kolmogorov's extension theorem,  $\nu$ can be extensively defined on the measurable space $([0,1), \mathcal{B})$. Now we list some facts about the expression for the measure of a cylinder.

(I) When $ n_i <n \leq n_i + \lfloor\phi(n_i)\rfloor~(i \geq 1)$, we have
\[
\nu(I(\varepsilon_1,\cdots,\varepsilon_{n})) = \nu(I(\varepsilon_1,\cdots,\varepsilon_{n_i})).
\]

(II) When $ n_i + j\lfloor\phi(n_i)\rfloor < n \leq n_i + j\lfloor\phi(n_i)\rfloor + \lfloor\phi(n_i)\rfloor -m$ ($i \geq 1$ and $1\leq j \leq k_i -1$), we deduce that
\[
\nu(I(\varepsilon_1,\cdots,\varepsilon_{n})) = \frac{|I(\varepsilon_1,\cdots,\varepsilon_{n})|}
{\beta^{-(n_i+j\lfloor\phi(n_i)\rfloor)}-\beta^{-(n_i+j\lfloor\phi(n_i)\rfloor+1)}}
\cdot \nu(I(\varepsilon_1,\cdots,\varepsilon_{n_i + (j-1)\lfloor\phi(n_i)\rfloor + \lfloor\phi(n_i)\rfloor -m})).
\]
In fact, the construction of $\mathcal{D}_n$ and the distribution of the measure $\nu$ yield that
\begin{equation}\label{ditui}
\nu(I(\varepsilon_1,\cdots,\varepsilon_{n})) = \frac{|I(\varepsilon_1,\cdots,\varepsilon_{n})|}{|I(\varepsilon_1,\cdots,\varepsilon_{n-1})|}\cdot
\nu(I(\varepsilon_1,\cdots,\varepsilon_{n-1})).
\end{equation}
Note that (\ref{ditui}) is also true if we replac $n$ by $(n-1)$, therefore, we obtain
\[
\nu(I(\varepsilon_1,\cdots,\varepsilon_{n}))=
\frac{|I(\varepsilon_1,\cdots,\varepsilon_{n})|}{|I(\varepsilon_1,\cdots,\varepsilon_{n-1})|}\cdot
\frac{|I(\varepsilon_1,\cdots,\varepsilon_{n-1})|}{|I(\varepsilon_1,\cdots,\varepsilon_{n-2})|}\cdot
\nu(I(\varepsilon_1,\cdots,\varepsilon_{n-2})).
\]
Repeating the above procedure, we finally have
\begin{equation}\label{ditui2}
\nu(I(\varepsilon_1,\cdots,\varepsilon_{n}))= \frac{|I(\varepsilon_1,\cdots,\varepsilon_{n})|}{|I(\varepsilon_1,\cdots,\varepsilon_{n_i + j[\phi(n_i)] +1})|}\cdot \nu (I(\varepsilon_1,\cdots,\varepsilon_{n_i + j\lfloor\phi(n_i)\rfloor +1})).
\end{equation}
It follows from $\varepsilon_{n_i + j\lfloor\phi(n_i)\rfloor +1} \neq 0$ and the distribution of the measure $\nu$ that
\begin{align}\label{ditui3}
\nu (I(\varepsilon_1,\cdots,\varepsilon_{n_i + j\lfloor\phi(n_i)\rfloor +1}))=& \frac{|I(\varepsilon_1,\cdots,\varepsilon_{n_i + j\lfloor\phi(n_i)\rfloor +1})|}{|I(\varepsilon_1,\cdots,\varepsilon_{n_i + j\lfloor\phi(n_i)\rfloor})\backslash I(\varepsilon_1,\cdots,\varepsilon_{n_i + j\lfloor\phi(n_i)\rfloor}, 0)|} \notag \\
&\times\nu (I(\varepsilon_1,\cdots,\varepsilon_{n_i + j\lfloor\phi(n_i)\rfloor}))
\end{align}
Since $\varepsilon_{n_i +(j-1)\lfloor\phi(n_i)\rfloor+ \lfloor\phi(n_i)\rfloor -m +1} = \cdots = \varepsilon_{n_i + j\lfloor\phi(n_i)\rfloor}=0$, by the construction of $\mathcal{D}_n$ and the distribution of the measure $\nu$, we obtain that
\begin{align}\label{ditui4}
\nu (I(\varepsilon_1,\cdots,\varepsilon_{n_i + j\lfloor\phi(n_i)\rfloor})) = \nu(I(\varepsilon_1,\cdots,\varepsilon_{n_i + (j-1)\lfloor\phi(n_i)\rfloor + \lfloor\phi(n_i)\rfloor -m})).
\end{align}
In view of Lemma \ref{full} and the construction of $\mathcal{D}_n$, we know that the two cylinders $I(\varepsilon_1,\cdots,\varepsilon_{n_i + j\lfloor\phi(n_i)\rfloor})$ and $I(\varepsilon_1,\cdots,\varepsilon_{n_i + j\lfloor\phi(n_i)\rfloor}, 0)$ are both full. Hence
\[
|I(\varepsilon_1,\cdots,\varepsilon_{n_i + j\lfloor\phi(n_i)\rfloor})\backslash I(\varepsilon_1,\cdots,\varepsilon_{n_i + j\lfloor\phi(n_i)\rfloor}, 0)| = \beta^{-(n_i+j\lfloor\phi(n_i)\rfloor)}-\beta^{-(n_i+j\lfloor\phi(n_i)\rfloor+1)}.
\]
Combing this with (\ref{ditui2}), (\ref{ditui3}) and (\ref{ditui4}), we get the desired result.

(III) When $ n_i + j\lfloor\phi(n_i)\rfloor + \lfloor\phi(n_i)\rfloor -m < n \leq n_i + (j+1)\lfloor\phi(n_i)\rfloor$ ($i \geq 1$ and $1\leq j \leq k_i -1$), it is easy to check
\[
\nu(I(\varepsilon_1,\cdots,\varepsilon_{n})) = \nu(I(\varepsilon_1,\cdots,\varepsilon_{n_i + j\lfloor\phi(n_i)\rfloor + \lfloor\phi(n_i)\rfloor -m})).
\]

(IV) When $n_i + k_i\lfloor\phi(n_i)\rfloor < n \leq n_{i+1}~(i \geq 1)$. In this case, we should distinguish two cases according to the relationship between the remainder $r_i$ and $m$.\\
{\bf(i)} If $r_i \leq m$, by the construction of $\mathcal{D}_n$ and the distribution of the measure $\nu$, we know that
\[
\nu(I(\varepsilon_1,\cdots,\varepsilon_{n}))  = \nu(I(\varepsilon_1,\cdots,\varepsilon_{n_i + (k_i-1)\lfloor\phi(n_i)\rfloor + \lfloor\phi(n_i)\rfloor -m})).
\]
{\bf(ii)} If $r_i > m$, we need to distinguish two cases according to the position of $n$.

(1) For $n_i + k_i\lfloor\phi(n_i)\rfloor < n \leq n_{i+1} -m -1$, being similar to (II), we obtain
\begin{align*}
\nu(I(\varepsilon_1,\cdots,\varepsilon_{n}))=& \frac{|I(\varepsilon_1,\cdots,\varepsilon_{n})|}
{\beta^{-(n_i+k_i\lfloor\phi(n_i)\rfloor)}-\beta^{-(n_i+k_i\lfloor\phi(n_i)\rfloor+1)}} \cdot
\nu(I(\varepsilon_1,\cdots,\varepsilon_{n_i + (k_i-1)\lfloor\phi(n_i)\rfloor + \lfloor\phi(n_i)\rfloor -m})).
\end{align*}

(2) For $n_{i+1}  -m -1 < n \leq n_{i+1}$, we get that
\[
\nu(I(\varepsilon_1,\cdots,\varepsilon_{n})) = \nu(I(\varepsilon_1,\cdots,\varepsilon_{n_{i+1}-m-1})).
\]

{\bf Step 3 Estimation on the $\nu$-measure of cylinders.} We claim that
\[
\nu(I(\varepsilon_1,\cdots,\varepsilon_{n_2})) \leq \frac{\beta^{(m+1-\log_\beta(\beta-1))(k_1+1)+ \lfloor\phi(n_1)\rfloor+r_1+1}}{\beta^{n_2 -n_1}}\cdot \nu(I(\varepsilon_1,\cdots,\varepsilon_{n_1})).
\]
We distinguish two cases to prove this statement according to the relation between $r_1$ and $m$.\\
$\bullet$ $r_1 \leq m$ \\
By the construction of $\mathcal{D}_n$, we know $\varepsilon_{n_1 +k_1\lfloor\phi(n_1)\rfloor-m +1} = \cdots = \varepsilon_{n_2 -1}=0$ and $\varepsilon_{n_2} \neq 0$ .
Since $|I(\varepsilon_1,\cdots,\varepsilon_{n_2})| \leq \beta^{-n_2}$,
being similar to the Case (II) in Step 2, we obtain
\begin{align}\label{estimate1}
\nu(I(\varepsilon_1,\cdots,\varepsilon_{n_2})) & = \frac{|I(\varepsilon_1,\cdots,\varepsilon_{n_2})|}{\beta^{n_2-1}-\beta^{n_2}}\cdot
\nu(I(\varepsilon_1,\cdots,\varepsilon_{n_1 + (k_1-1)\lfloor\phi(n_1)\rfloor + \lfloor\phi(n_1)\rfloor -m}))\notag \\
&\leq \frac{1}{\beta -1}\cdot \nu(I(\varepsilon_1,\cdots,\varepsilon_{n_1 + (k_1-1)\lfloor\phi(n_1)\rfloor + \lfloor\phi(n_1)\rfloor -m})).
\end{align}
Now it remains to estimate the measure $\nu(I(\varepsilon_1,\cdots,\varepsilon_{n_1 + (k_1-1)\lfloor\phi(n_1)\rfloor + \lfloor\phi(n_1)\rfloor -m}))$. Being similar to the Case (II) in Step 2, we deduce that
\begin{align}\label{estimate2}
&\nu(I(\varepsilon_1,\cdots,\varepsilon_{n_1 + (k_1-1)\lfloor\phi(n_1)\rfloor + \lfloor\phi(n_1)\rfloor -m}))\notag \\
=& \frac{|I(\varepsilon_1,\cdots,\varepsilon_{n_1 + (k_1-1)\lfloor\phi(n_1)\rfloor + \lfloor\phi(n_1)\rfloor -m})|}{\beta^{-(n_1 + (k_1-1)\lfloor\phi(n_1)\rfloor)}-\beta^{-(n_1 + (k_1-1)\lfloor\phi(n_1)\rfloor+1)}}\cdot \nu(I(\varepsilon_1,\cdots,\varepsilon_{n_1 + (k_1-2)\lfloor\phi(n_1)\rfloor + \lfloor\phi(n_1)\rfloor -m}))\notag \\
\leq & \frac{1}{\beta -1} \cdot \frac{1}{\beta^{\lfloor\phi(n_1)\rfloor-m-1}} \cdot \nu(I(\varepsilon_1,\cdots,\varepsilon_{n_1 + (k_1-2)\lfloor\phi(n_1)\rfloor + \lfloor\phi(n_1)\rfloor -m}))
\end{align}
since $|I(\varepsilon_1,\cdots,\varepsilon_{n_1 + (k_1-1)\lfloor\phi(n_1)\rfloor + \lfloor\phi(n_1)\rfloor -m})| \leq \beta^{-(n_1 + (k_1-1)\lfloor\phi(n_1)\rfloor + \lfloor\phi(n_1)\rfloor -m)}$.
Repeating the above procedure, we have that
\begin{equation}\label{diedai}
\nu(I(\varepsilon_1,\cdots,\varepsilon_{n_1 + (k_1-1)\lfloor\phi(n_1)\rfloor + \lfloor\phi(n_1)\rfloor -m})) \leq \left(\frac{1}{\beta -1} \cdot \frac{1}{\beta^{\lfloor\phi(n_1)\rfloor-m-1}}\right)^{k_1-1}\cdot \nu(I(\varepsilon_1,\cdots,\varepsilon_{n_1})).
\end{equation}
Combining this with (\ref{estimate1}) and (\ref{estimate2}), we obtain that
\begin{align*}
\nu(I(\varepsilon_1,\cdots,\varepsilon_{n_2})) &\leq \frac{1}{\beta -1} \cdot \left(\frac{1}{\beta -1} \cdot \frac{1}{\beta^{\lfloor\phi(n_1)\rfloor-m-1}}\right)^{k_1-1}\cdot \nu(I(\varepsilon_1,\cdots,\varepsilon_{n_1}))\\
& \leq   \frac{\beta^{k_1(m+1-\log_\beta(\beta-1))+\lfloor\phi(n_1)\rfloor}}{\beta^{k_1\lfloor\phi(n_1)\rfloor}}\cdot \nu(I(\varepsilon_1,\cdots,\varepsilon_{n_1})).
\end{align*}
Note that $m+1-\log_\beta(\beta-1) >0$ and $k_1\lfloor\phi(n_1)\rfloor = n_2-n_1-r_1-1$, we have
\[
\nu(I(\varepsilon_1,\cdots,\varepsilon_{n_2})) \leq \frac{\beta^{(m+1-\log_\beta(\beta-1))(k_1+1)+ \lfloor\phi(n_1)\rfloor+r_1+1}}{\beta^{n_2 -n_1}}\cdot \nu(I(\varepsilon_1,\cdots,\varepsilon_{n_1})).
\]
$\bullet$ $r_1 >m$.\\
By the construction of $\mathcal{D}_n$, we know $\varepsilon_{n_2} \neq 0$  and $\varepsilon_{n_2 -m } = \cdots = \varepsilon_{n_2 -1}=0$. So,
\begin{align}\label{estimate3}
\nu(I(\varepsilon_1,\cdots,\varepsilon_{n_2}))  =\frac{|I(\varepsilon_1,\cdots,\varepsilon_{n_2})|}{\beta^{n_2-1}-\beta^{n_2}} \cdot
\nu(I(\varepsilon_1,\cdots,\varepsilon_{n_2 -m-1})).
\end{align}
Notice that $\varepsilon_{n_2 -m-1} \neq 0$, by the distribution of the measure $\nu$, we deduce that
\begin{align}\label{estimate4}
\nu(I(\varepsilon_1,\cdots,\varepsilon_{n_2 -m-1})) =& \frac{|I(\varepsilon_1,\cdots,\varepsilon_{n_2 -m-1})|}
{\beta^{-(n_1+k_1\lfloor\phi(n_1)\rfloor)}- \beta^{-(n_1+k_1\lfloor\phi(n_1)\rfloor+1)}} \notag\\
& \times \nu(I(\varepsilon_1,\cdots,\varepsilon_{n_1 + (k_1-1)\lfloor\phi(n_1)\rfloor + \lfloor\phi(n_1)\rfloor -m})).
\end{align}
It follows form (\ref{estimate3}) and (\ref{estimate4}) that
\[
\nu(I(\varepsilon_1,\cdots,\varepsilon_{n_2})) \leq \left(\frac{1}{\beta-1}\right)^2 \cdot \nu(I(\varepsilon_1,\cdots,\varepsilon_{n_1 + (k_1-1)\lfloor\phi(n_1)\rfloor + \lfloor\phi(n_1)\rfloor -m})).
\]
Combining this with (\ref{diedai}), we deduce that
\[
\nu(I(\varepsilon_1,\cdots,\varepsilon_{n_2})) \leq \left(\frac{1}{\beta-1}\right)^2 \cdot
\left(\frac{1}{\beta -1} \cdot \frac{1}{\beta^{\lfloor\phi(n_1)\rfloor-m-1}}\right)^{k_1-1}\cdot \nu(I(\varepsilon_1,\cdots,\varepsilon_{n_1})).
\]
Since $m+1-\log_\beta(\beta-1) >0$ and $k_1\lfloor\phi(n_1)\rfloor = n_2-n_1-r_1-1$, we obtain the desired result by following the calculations at the end of the case of $r_1  \leq m$.

As mentioned above, no matter $r_1 \leq m$ and $r_1 > m$, we always have
\[
\nu(I(\varepsilon_1,\cdots,\varepsilon_{n_2})) \leq \frac{\beta^{(m+1-\log_\beta(\beta-1))(k_1+1)+ \lfloor\phi(n_1)\rfloor+r_1+1}}{\beta^{n_2 -n_1}}\cdot \nu(I(\varepsilon_1,\cdots,\varepsilon_{n_1})).
\]

More generally, we have the following lemma whose proof is similar to the above arguments.

\begin{lemma}\label{DIE}
For any $i \geq 1$,
\[
\nu(I(\varepsilon_1,\cdots,\varepsilon_{n_i})) \leq \beta^{U_i - n_i},
\]
where $U_1 :=0$ and
\[
U_i =  \big(m+1- \log_{\beta}(\beta-1)\big)\sum_{j=1}^{i-1}(k_j +1)+ \sum_{j=1}^{i-1}\big(\lfloor\phi(n_j)\rfloor+r_j+1\big)\ \text{for any}\ i \geq 2.
\]
\end{lemma}

By Lemma \ref{DIE}, a simple calculation implies the following result.

\begin{lemma}\label{DIEDAI}
For any $i \geq 1$ and $1\leq j \leq k_i-1$,
\[
\nu(I(\varepsilon_1,\cdots,\varepsilon_{n_i + (j-1)\lfloor\phi(n_i)\rfloor + \lfloor\phi(n_i)\rfloor -m})) \leq \frac{\beta^{(m+1-\log_\beta(\beta-1))(j-1)+U_i}}{\beta^{n_i+(j-1)\lfloor\phi(n_i)\rfloor}},
\]
where $U_i$ is as defined in Lemma \ref{DIE}.
\end{lemma}

\begin{proof}
Since $\varepsilon_{n_i + (j-1)\lfloor\phi(n_i)\rfloor +1} \neq 0$, by the definition of the measure $\nu$, we deduce
\begin{align*}
\nu(I(\varepsilon_1,\cdots,\varepsilon_{n_i + (j-1)\lfloor\phi(n_i)\rfloor + \lfloor\phi(n_i)\rfloor -m}))=&
\frac{|I(\varepsilon_1,\cdots,\varepsilon_{n_i + (j-1)\lfloor\phi(n_i)\rfloor + \lfloor\phi(n_i)\rfloor -m})|}
{\beta^{-(n_i + (j-1)\lfloor\phi(n_i)\rfloor)}- \beta^{-(n_i + (j-1)\lfloor\phi(n_i)\rfloor+1)}} \\
&\times \nu(I(\varepsilon_1,\cdots,\varepsilon_{n_i + (j-2)\lfloor\phi(n_i)\rfloor + \lfloor\phi(n_i)\rfloor -m})).
\end{align*}
Therefore,
\begin{align*}
\nu(I(\varepsilon_1,\cdots,\varepsilon_{n_i + (j-1)\lfloor\phi(n_i)\rfloor + [\phi(n_i)] -m}))\leq & \frac{1}{\beta-1}\cdot \frac{1}{\beta^{\lfloor\phi(n_i)\rfloor-m-1}} \cdot
\nu(I(\varepsilon_1,\cdots,\varepsilon_{n_i + (j-2)\lfloor\phi(n_i)\rfloor + \lfloor\phi(n_i)\rfloor -m})).
\end{align*}
Repeating the above procedure, we obtain that
\[
\nu(I(\varepsilon_1,\cdots,\varepsilon_{n_i + (j-1)\lfloor\phi(n_i)\rfloor + \lfloor\phi(n_i)\rfloor -m})) \leq
\left(\frac{1}{\beta-1}\cdot \frac{1}{\beta^{\lfloor\phi(n_i)\rfloor-m-1}}\right)^{j-1} \cdot \nu(I(\varepsilon_1,\cdots,\varepsilon_{n_i})).
\]
By Lemma \ref{DIE}, we complete the proof.
\end{proof}

Next we will check the inequality in Proposition \ref{MMDP} according to the classification of $n$ at the end of Step~2.
That is,
\begin{equation}\label{MMDP inequality}
\nu(I(\varepsilon_1,\cdots,\varepsilon_{n})) \leq C_\delta \cdot|(\varepsilon_1,\cdots,\varepsilon_{n})|^{s-\delta},
\end{equation}
where $C_\delta>0$ is a constant only depending on $\delta$.

(I) When $ n_i <n \leq n_i + \lfloor\phi(n_i)\rfloor~(i \geq 1)$, by Lemma \ref{DIE},  we have that
\[
\nu(I(\varepsilon_1,\cdots,\varepsilon_{n})) = \nu(I(\varepsilon_1,\cdots,\varepsilon_{n_i})) \leq \beta^{U_i - n_i} \leq \beta^{\frac{\delta}{2}(n_i+\lfloor\phi(n_i)\rfloor)-n_i},
\]
where the last inequality is from (\ref{tiaojian}). It follows from Proposition \ref{A0} that $$C_0\beta^{-n} \leq |(\varepsilon_1,\cdots,\varepsilon_{n})| \leq \beta^{-n}$$ and hence
\begin{align*}
\frac{\nu(I(\varepsilon_1,\cdots,\varepsilon_{n}))}{|(\varepsilon_1,\cdots,\varepsilon_{n}|^{s-\delta}}
\leq C^{\delta -s}_0 \beta^{n(s-\delta) +\frac{\delta}{2}(n_i+\lfloor\phi(n_i)\rfloor) - n_i} \leq C^{\delta -s}_0 \beta^{n_it_i},
\end{align*}
where $t_i = (1+\lfloor\phi(n_i)\rfloor/n_i)(s-\delta/2)-1$ and the last inequality is because $n \leq n_i + \lfloor\phi(n_i)\rfloor$. By the definition of $s$ and $t_i$, we know that $\lim\limits_{i\to \infty} t_i = -\delta/(2s) \leq -\delta/2$ for (\ref{zilie}). So there exists a constant $C_1$ such that $\beta^{n_it_i} \leq C_1$ for all $i \geq 1$. Let $C_\delta = C_1\cdot C^{\delta -s}_0$. Then we obtain that (\ref{MMDP inequality}) holds for the measure $\nu$ and such a constant $C_\delta$.

(II) When $ n_i + j\lfloor\phi(n_i)\rfloor < n \leq n_i + j\lfloor\phi(n_i)\rfloor+ \lfloor\phi(n_i)\rfloor -m$ ($i \geq 1$ and $1\leq j \leq k_i -1$), we deduce that
\[
\nu(I(\varepsilon_1,\cdots,\varepsilon_{n})) = \frac{|I(\varepsilon_1,\cdots,\varepsilon_{n})|}
{\beta^{-(n_i+j\lfloor\phi(n_i)\rfloor)}-\beta^{-(n_i+j\lfloor\phi(n_i)\rfloor+1)}}
\cdot \nu(I(\varepsilon_1,\cdots,\varepsilon_{n_i + (j-1)\lfloor\phi(n_i)\rfloor + \lfloor\phi(n_i)\rfloor -m})).
\]
In view of  Lemma \ref{DIEDAI}, we have that
\begin{align*}
\nu(I(\varepsilon_1,\cdots,\varepsilon_{n})) &\leq \frac{\beta^{n_i+j\lfloor\phi(n_i)\rfloor+1}}{\beta-1} \cdot \frac{1}{\beta^n} \cdot \frac{\beta^{(m+1-\log_\beta(\beta-1))(j-1)+U_i}}{\beta^{n_i+(j-1)\lfloor\phi(n_i)\rfloor}}\\
& = \frac{\beta^{-n}}{\beta-1} \cdot \beta^{(m+1-\log_\beta(\beta-1))(j-1)+U_i+\lfloor\phi(n_i)\rfloor+1}.
\end{align*}
Therefore,
\begin{align}\label{II}
\frac{\nu(I(\varepsilon_1,\cdots,\varepsilon_{n}))}{|(\varepsilon_1,\cdots,\varepsilon_{n}|^{s-\delta}}
\leq \frac{C^{\delta-s}_0}{\beta-1}\cdot \beta^{n(s-\delta-1)} \cdot \beta^{(m+1-\log_\beta(\beta-1))(j-1)+U_i+\lfloor\phi(n_i)\rfloor+1}.
\end{align}
Note that $s-\delta-1<0$ and $n_i + j\lfloor\phi(n_i)\rfloor < n \leq n_i + j\lfloor\phi(n_i)\rfloor + \lfloor\phi(n_i)\rfloor -m$, so
\begin{align*}
&n(s-\delta-1)+ (m+1-\log_\beta(\beta-1))(j-1)+U_i+\lfloor\phi(n_i)\rfloor+1 \\
\leq& (n_i + j\lfloor\phi(n_i)\rfloor)(s-\delta-1)+ (m+1-\log_\beta(\beta-1))(j-1)+U_i+\lfloor\phi(n_i)\rfloor+1\\
= & (n_i + \lfloor\phi(n_i)\rfloor)(s-\delta-1) + (j-1)W _i + U_i +\lfloor\phi(n_i)\rfloor+1\\
\leq & (n_i + \lfloor\phi(n_i)\rfloor)(s-\delta/2-1)+\lfloor\phi(n_i)\rfloor+1,
\end{align*}
where $W_i = \lfloor\phi(n_i)\rfloor(s-\delta-1) + m+1-\log_{\beta}(\beta-1) <0$ and the last inequality follows from
the condition (\ref{tiaojian}), i.e., $U_i \leq \delta(n_i + \lfloor\phi(n_i)\rfloor)/2$. Combining this with (\ref{II}), we obtain that
\[
\frac{\nu(I(\varepsilon_1,\cdots,\varepsilon_{n}))}{|(\varepsilon_1,\cdots,\varepsilon_{n}|^{s-\delta}} \leq \frac{C^{\delta-s}_0}{\beta-1}\cdot \beta^{n_it_i},
\]
where $t_i = (1+\lfloor\phi(n_i)\rfloor/n_i)(s-\delta/2-1)+ (\lfloor\phi(n_i)\rfloor+1)/n_i$. By (\ref{zilie}) and the definition of $s$, we know that $\lim\limits_{i\to \infty}t_i = -\delta/(2s) \leq -\delta/2$. The similar methods at the end of Case (I) assure that (\ref{MMDP inequality}) holds.

(III) When $ n_i + j\lfloor\phi(n_i)\rfloor + \lfloor\phi(n_i)\rfloor -m < n \leq n_i + (j+1)\lfloor\phi(n_i)\rfloor$ ($i \geq 1$ and $1\leq j \leq k_i -1$), then
\[
\nu(I(\varepsilon_1,\cdots,\varepsilon_{n})) = \nu(I(\varepsilon_1,\cdots,\varepsilon_{n_i + j\lfloor\phi(n_i)\rfloor + \lfloor\phi(n_i)\rfloor -m})).
\]
It follows from Lemma \ref{DIEDAI} that
\[
\nu(I(\varepsilon_1,\cdots,\varepsilon_{n_i + j\lfloor\phi(n_i)\rfloor + \lfloor\phi(n_i)\rfloor -m})) \leq
\frac{\beta^{j(m+1-\log_\beta(\beta-1))+U_i}}{\beta^{n_i+j\lfloor\phi(n_i)\rfloor}}.
\]
Therefore,
\begin{align}\label{III}
\frac{\nu(I(\varepsilon_1,\cdots,\varepsilon_{n}))}{|(\varepsilon_1,\cdots,\varepsilon_{n}|^{s-\delta}} \leq C^{\delta-s}_0 \beta^{n(s-\delta)} \cdot \frac{\beta^{j(m+1-\log_\beta(\beta-1))+U_i}}{\beta^{n_i+j\lfloor\phi(n_i)\rfloor}}.
\end{align}
Since $s-\delta >0$ and $n_i + j\lfloor\phi(n_i)\rfloor + \lfloor\phi(n_i)\rfloor -m< n \leq n_i +(j+1)\lfloor\phi(n_i)\rfloor$, we have
\begin{align*}
&n(s-\delta)+ j(m+1-\log_\beta(\beta-1))+U_i - \left(n_i+j\lfloor\phi(n_i)\rfloor\right)\\
\leq &(n_i +(j+1)\lfloor\phi(n_i)\rfloor)(s-\delta) + j(m+1-\log_\beta(\beta-1))+U_i -\left(n_i+j\lfloor\phi(n_i)\rfloor\right)\\
=& (n_i +\lfloor\phi(n_i)\rfloor)(s-\delta)+jW_i+U_i -n_i \leq (n_i +\lfloor\phi(n_i)\rfloor)(s-\delta/2)-n_i.
\end{align*}
Combining this with (\ref{III}), we obtain that
\[
\frac{\nu(I(\varepsilon_1,\cdots,\varepsilon_{n}))}{|(\varepsilon_1,\cdots,\varepsilon_{n}|^{s-\delta}} \leq C^{\delta-s}_0  \beta^{n_it_i},
\]
where $t_i = (1+\lfloor\phi(n_i)\rfloor/n_i)(s-\delta/2)-1$. By (\ref{zilie}) and the definition of $s$, we know that $\lim\limits_{i\to \infty}t_i = -\delta/(2s) \leq -\delta/2$. The similar methods at the end of Case (I) guarantee that (\ref{MMDP inequality}) holds.

(IV) Let $n_i + k_i\lfloor\phi(n_i)\rfloor < n \leq n_{i+1}~(i \geq 1)$.\\
{\bf(i)} If $r_i \leq m$, we know that
\[
\nu(I(\varepsilon_1,\cdots,\varepsilon_{n}))  = \nu(I(\varepsilon_1,\cdots,\varepsilon_{n_i + (k_i-1)\lfloor\phi(n_i)\rfloor + \lfloor\phi(n_i)\rfloor -m})).
\]
It follows from Lemma \ref{DIEDAI} that
\[
\nu(I(\varepsilon_1,\cdots,\varepsilon_{n_i + (k_i-1)[\phi(n_i)] + \lfloor\phi(n_i)\rfloor -m})) \leq
\frac{\beta^{(m+1-\log_\beta(\beta-1))(k_i-1)+U_i}}{\beta^{n_i+(k_i-1)\lfloor\phi(n_i)\rfloor}}.
\]
Therefore,
\begin{align}\label{IVi}
\frac{\nu(I(\varepsilon_1,\cdots,\varepsilon_{n}))}{|(\varepsilon_1,\cdots,\varepsilon_{n}|^{s-\delta}} \leq C^{\delta-s}_0 \beta^{n(s-\delta)} \cdot \frac{\beta^{(m+1-\log_\beta(\beta-1))(k_i-1)+U_i}}{\beta^{n_i+(k_i-1)\lfloor\phi(n_i)\rfloor}}.
\end{align}
Notice that $n_i + k_i\lfloor\phi(n_i)\rfloor < n \leq n_{i+1}~(i \geq 1)$ and  $r_i \leq m$, by the definition of $k_i$, we know that $n_{i+1} = n_i+k_i\lfloor\phi(n_i)\rfloor+r_i+1 \leq n_i+k_i\lfloor\phi(n_i)\rfloor+m+1$. So,
\begin{align*}
&n(s-\delta)+ (m+1-\log_\beta(\beta-1))(k_i-1)+U_i - \left(n_i+(k_i-1)\lfloor\phi(n_i)\rfloor\right)\\
\leq &(n_i +k_i\lfloor\phi(n_i)\rfloor+m+1)(s-\delta) + (m+1-\log_\beta(\beta-1))(k_i-1)
+U_i -\left(n_i+(k_i-1)\lfloor\phi(n_i)\rfloor\right) \\
=& (n_i +\lfloor\phi(n_i)\rfloor)(s-\delta)+ (m+1)(s-\delta)+(k_i-1)W_i+ U_i -n_i \\
\leq& (n_i +\lfloor\phi(n_i)\rfloor)(s-\delta/2)-n_i + (m+1)(s-\delta),
\end{align*}
Combining this with (\ref{IVi}), we have
\[
\frac{\nu(I(\varepsilon_1,\cdots,\varepsilon_{n}))}{|(\varepsilon_1,\cdots,\varepsilon_{n}|^{s-\delta}} \leq C^{\delta-s}_0  \beta^{n_it_i},
\]
where $t_i = (1+\lfloor\phi(n_i)\rfloor/n_i)(s-\delta/2)-1+(m+1)(s-\delta)/n_i$. By (\ref{zilie}) and the definition of $s$, we know that $\lim\limits_{i\to \infty}t_i = -\delta/(2s) \leq -\delta/2$ and hence that (\ref{MMDP inequality}) holds.\\
{\bf (ii)} Let $r_i >m$.

(1) For $n_i + k_i\lfloor\phi(n_i)\rfloor < n \leq n_{i+1} -m -1$, we obtain that
\begin{align*}
\nu(I(\varepsilon_1,\cdots,\varepsilon_{n}))=& \frac{|I(\varepsilon_1,\cdots,\varepsilon_{n})|}
{\beta^{-(n_i+k_i\lfloor\phi(n_i)\rfloor)}-\beta^{-(n_i+k_i\lfloor\phi(n_i)\rfloor+1)}} \cdot
\nu(I(\varepsilon_1,\cdots,\varepsilon_{n_i + (k_i-1)\lfloor\phi(n_i)\rfloor + \lfloor\phi(n_i)\rfloor -m})).
\end{align*}
It follows from Lemma \ref{DIEDAI} that
\[
\nu(I(\varepsilon_1,\cdots,\varepsilon_{n_i + (k_i-1)\lfloor\phi(n_i)\rfloor + \lfloor\phi(n_i)\rfloor -m})) \leq
\frac{\beta^{(m+1-\log_\beta(\beta-1))(k_i-1)+U_i}}{\beta^{n_i+(k_i-1)\lfloor\phi(n_i)\rfloor}}.
\]
So,
\begin{align*}
\nu(I(\varepsilon_1,\cdots,\varepsilon_{n})) \leq&
\frac{\beta^{n_i+k_i\lfloor\phi(n_i)\rfloor+1}}{\beta -1}\cdot \frac{1}{\beta^n}\cdot
\frac{\beta^{(m+1-\log_\beta(\beta-1))(k_i-1)+U_i}}{\beta^{n_i+(k_i-1)\lfloor\phi(n_i)\rfloor}}\notag\\
= & \frac{\beta^{-n}}{\beta -1} \cdot \beta^{(m+1-\log_\beta(\beta-1))(k_i-1)+U_i+\lfloor\phi(n_i)\rfloor+1}.
\end{align*}
Therefore,
\begin{align}\label{IVii(1)}
\frac{\nu(I(\varepsilon_1,\cdots,\varepsilon_{n}))}{|(\varepsilon_1,\cdots,\varepsilon_{n}|^{s-\delta}} \leq  \frac{C^{\delta-s}_0}{\beta -1} \cdot \beta^{n(s-\delta-1)} \cdot \beta^{(m+1-\log_\beta(\beta-1))(k_i-1)+U_i+\lfloor\phi(n_i)\rfloor+1}.
\end{align}
Note that $s-\delta-1 <0$ and $n_i + k_i\lfloor\phi(n_i)\rfloor < n \leq n_{i+1} -m -1$, we deduce
\begin{align*}
&n(s-\delta-1)+(m+1-\log_\beta(\beta-1))(k_i-1)+U_i+\lfloor\phi(n_i)\rfloor+1\\
\leq & (n_i + k_i\lfloor\phi(n_i)\rfloor)(s-\delta-1)+ (m+1-\log_\beta(\beta-1))(k_i-1)+U_i+\lfloor\phi(n_i)\rfloor+1 \\
\leq & (n_i + \lfloor\phi(n_i)\rfloor)(s-\delta-1)+ (k_i-1)W_i + U_i+\lfloor\phi(n_i)\rfloor+1 \\
\leq & (n_i + \lfloor\phi(n_i)\rfloor)(s-\delta/2-1)+\lfloor\phi(n_i)\rfloor+1.
\end{align*}
Combining this with (\ref{IVii(1)}), we obtain that
\[
\frac{\nu(I(\varepsilon_1,\cdots,\varepsilon_{n}))}{|(\varepsilon_1,\cdots,\varepsilon_{n}|^{s-\delta}} \leq \frac{C^{\delta-s}_0}{\beta-1}\cdot  \beta^{n_it_i},
\]
where  $t_i = (1+\lfloor\phi(n_i)\rfloor/n_i)(s-\delta/2-1)+ (\lfloor\phi(n_i)\rfloor+1)/n_i$. By (\ref{zilie}) and the definition of $s$, we know that $\lim\limits_{i\to \infty}t_i = -\delta/(2s) \leq -\delta/2$ and hence that (\ref{MMDP inequality}) holds.

(2) For $n_{i+1}  -m -1 < n \leq n_{i+1}$, we get that
\[
\nu(I(\varepsilon_1,\cdots,\varepsilon_{n})) = \nu(I(\varepsilon_1,\cdots,\varepsilon_{n_{i+1}-m-1})).
\]
Since $\varepsilon_{n_i + k_i\lfloor\phi(n_i)\rfloor +1} \neq 0$, by the distribution of the measure $\nu$, we deduce
\begin{align*}
\nu(I(\varepsilon_1,\cdots,\varepsilon_{n_{i+1}-m-1})) =& \frac{|I(\varepsilon_1,\cdots,\varepsilon_{n_{i+1}-m-1})|}{\beta^{-(n_i+k_i\lfloor\phi(n_i)\rfloor)}-
\beta^{-(n_i+k_i\lfloor\phi(n_i)\rfloor+1)}} \cdot
 \nu(I(\varepsilon_1,\cdots,\varepsilon_{n_i + (k_i-1)\lfloor\phi(n_i)\rfloor + \lfloor\phi(n_i)\rfloor -m})).
\end{align*}
Therefore,
\begin{align}\label{IVii(21)}
\frac{\nu(I(\varepsilon_1,\cdots,\varepsilon_{n_{i+1}-m-1}))}{\nu(I(\varepsilon_1,\cdots,\varepsilon_{n_i + (k_i-1)\lfloor\phi(n_i)\rfloor + \lfloor\phi(n_i)\rfloor -m}))}\leq &\frac{\beta^{n_i+k_i\lfloor\phi(n_i)\rfloor+1}}{\beta-1} \cdot
\frac{1}{\beta^{n_{i+1}-m-1}}
=\frac{1}{\beta-1}\cdot \frac{1}{\beta^{r_i-m-1}},
\end{align}
where the last equation is from the definition of $n_{i+1}$, i.e., $n_{i+1} = n_i + k_i\lfloor\phi(n_i)\rfloor+r_i +1$.
Notice that
\[
\nu(I(\varepsilon_1,\cdots,\varepsilon_{n_i + (k_i-1)\lfloor\phi(n_i)\rfloor + \lfloor\phi(n_i)\rfloor -m})) \leq
\frac{\beta^{(m+1-\log_\beta(\beta-1))(k_i-1)+U_i}}{\beta^{n_i+(k_i-1)\lfloor\phi(n_i)\rfloor}},
\]
combining this with (\ref{IVii(21)}), we obtain that
\[
\nu(I(\varepsilon_1,\cdots,\varepsilon_{n_{i+1}-m-1}))\leq  \frac{1}{\beta-1}\cdot \frac{1}{\beta^{r_i-m-1}} \cdot \frac{\beta^{(m+1-\log_\beta(\beta-1))(k_i-1)+U_i}}{\beta^{n_i+(k_i-1)\lfloor\phi(n_i)\rfloor}}.
\]
Therefore,
\begin{align*}
\frac{\nu(I(\varepsilon_1,\cdots,\varepsilon_{n}))}{|(\varepsilon_1,\cdots,\varepsilon_{n}|^{s-\varepsilon}} \leq & \frac{C^{\delta-s}_0}{\beta -1} \cdot \frac{\beta^{n(s-\delta)}}{\beta^{r_i-m-1}} \cdot \frac{\beta^{(m+1-\log_\beta(\beta-1))(k_i-1)+U_i}}{\beta^{n_i+(k_i-1)\lfloor\phi(n_i)\rfloor}}\notag\\
\leq & \frac{C^{\delta-s}_0}{\beta -1} \cdot \frac{\beta^{(n_i+k_i\lfloor\phi(n_i)\rfloor+r_i+1)(s-\delta)+(m+1-\log_\beta(\beta-1))(k_i-1)+U_i}}
{\beta^{n_i+(k_i-1)\lfloor\phi(n_i)\rfloor+l_i-m-1}}\notag\\
= &\frac{C^{\delta-s}_0}{\beta -1} \cdot \beta^{(n_i+\lfloor\phi(n_i)\rfloor)(s-\delta)+ (r_i+1)(s-\delta-1)+(k_i-1)W_i + U_i -n_i +(m+2)}\notag\\
\leq & \frac{C^{\delta-s}_0}{\beta -1} \cdot \beta^{(n_i+\lfloor\phi(n_i)\rfloor)(s-\delta/2) -n_i +(m+2)} = \frac{C^{\delta-s}_0}{\beta -1} \cdot \beta^{n_it_i},
\end{align*}
where $t_i = (1+\lfloor\phi(n_i)\rfloor/n_i)(s-\delta/2)-1 + (m+2)/n_i$, the second inequality follows from $s-\delta >0$ and $n \leq n_{i+1} = n_i+k_i\lfloor\phi(n_i)\rfloor+r_i+1$ and the last inequality is from the condition (\ref{tiaojian}), $s-\delta-1<0$ and $W_i<0$.  By (\ref{zilie}) and the definition of $s$, we know that $\lim\limits_{i\to \infty}t_i = -\delta/(2s) \leq -\delta/2$ and hence that (\ref{MMDP inequality}) holds.

Applying Proposition \ref{MMDP} to the set $F$ and the measure $\nu$, we have that
\[
\dim_{\rm H} F_\phi \geq \dim_{\rm H} F \geq s-\delta.
\]
Letting $\delta\to 0^+$, we have $\dim_{\rm H} F_\phi \geq s$. We eventually obtain
\[
\dim_{\rm H} F_\phi \geq \frac{1}{1+\lim\limits_{n \to \infty}\phi(n)/n}.
\]

\subsubsection{General case for any $\beta>1$}

Now we will extend the result of Lemma \ref{F} from $\beta \in A_0$ to any $\beta >1$ by using the approximation method for the $\beta$-shift in Section \ref{Approximation method}. Let $\beta >\beta^\prime >1$ and $\beta^\prime \in A_0$. Here we denote $\ell_n(x,\beta)$ and $\ell_n(x,\beta^\prime)$ the length of the longest string of zeros just after the $n$-th digit in the $\beta$-expansion and $\beta^\prime$-expansion of $x$ respectively.
Lemma \ref{F} has showed that
\begin{equation}\label{F prime}
\dim_{\rm H}\left\{x \in [0,1): \limsup_{n \to \infty} \frac{\ell_n(x,\beta^\prime)}{\phi(n)} = 1 \right\} \geq
\frac{1}{1+\liminf\limits_{n \to \infty} \phi(n)/n}
\end{equation}
under the assumption that $\phi$ is a positive and nondecreasing function with $\phi(n) \to \infty$ as $n \to \infty$.

\begin{lemma}\label{F general}
Let $\beta >1$ be a real number. Assume that $\phi$ is a positive and nondecreasing function defined on $\mathbb{N}$ with $\phi(n) \to \infty$ as $n \to \infty$. Then
\begin{equation*}
\dim_{\rm H} F_\phi \geq \frac{1}{1+\liminf\limits_{n \to \infty} \phi(n)/n}.
\end{equation*}
\end{lemma}

\begin{proof}
Let $\beta >\beta^\prime >1$ and $\beta^\prime \in A_0$.
Since $H_\beta^{\beta^\prime}:= \pi_\beta(\Sigma_{\beta^\prime})$ is a Cantor set of $[0,1)$, we have that
\begin{equation}\label{1}
\dim_{\rm H}\left\{x \in [0,1): \limsup_{n \to \infty} \frac{\ell_n(x,\beta)}{\phi(n)} = 1 \right\} \geq
\dim_{\rm H}\left\{x \in H_\beta^{\beta^\prime}: \limsup_{n \to \infty} \frac{\ell_n(x,\beta)}{\phi(n)} = 1 \right\}.
\end{equation}
It follows from Proposition \ref{h} $(iv)$ that the function $h$ is H\"{o}lder continuous and hence that
\begin{align}\label{2}
& \dim_{\rm H} \left\{x \in H_\beta^{\beta^\prime}: \limsup_{n \to \infty} \frac{\ell_n(x,\beta)}{\phi(n)} = 1 \right\}
\geq \frac{\log \beta^\prime}{\log \beta} \cdot \dim_H h\left(\left\{x \in H_\beta^{\beta^\prime}: \limsup_{n \to \infty} \frac{\ell_n(x,\beta)}{\phi(n)} = 1 \right\}\right).
\end{align}
Note that $\varepsilon(h(x),\beta^\prime)= \varepsilon(x,\beta)$ for any $x \in H_\beta^{\beta^\prime}$, we deduce that $\ell_n(h(x),\beta^\prime)= \ell_n(x,\beta)$ by the definition of $\ell_n(x,\beta)$ defined in Section 3. Since $h$ is bijective, we obtain
\[
h\left(\left\{x \in H_\beta^{\beta^\prime}: \limsup_{n \to \infty} \frac{\ell_n(x,\beta)}{\phi(n)} = 1 \right\}\right) =
\left\{y \in [0,1): \limsup_{n \to \infty} \frac{\ell_n(y,\beta^\prime)}{\phi(n)} = 1 \right\}.
\]
Combining this with (\ref{1}) and (\ref{2}), we have that
\[
\dim_{\rm H} F_\phi \geq \frac{\log \beta^\prime}{\log \beta}\cdot \dim_{\rm H} \left\{y \in [0,1): \limsup_{n \to \infty} \frac{\ell_n(y,\beta^\prime)}{\phi(n)} = 1 \right\} \geq \frac{\log \beta^\prime}{\log \beta}\cdot
\frac{1}{1+\liminf\limits_{n \to \infty} \phi(n)/n},
\]
where the last inequality is from (\ref{F prime}). At last, let $\beta^\prime \to \beta$, we complete the proof since $A_0$ is dense in $(1,+\infty)$.
\end{proof}

By the upper bound Hausdorff dimension of $F_\phi$ in Section 4.1, we have

\begin{proposition}\label{F quan}
Let $\beta >1$ be a real number.
Assume that $\phi$ is a positive and nondecreasing function defined on $\mathbb{N}$ with $\phi(n) \to \infty$ as $n \to \infty$. Then
\begin{equation*}
\dim_{\rm H} F_\phi = \frac{1}{1+\liminf\limits_{n \to \infty} \phi(n)/n}.
\end{equation*}
\end{proposition}

For any $\beta \in A_0$, being similar to the Lemma \ref{F}, we can construct a Cantor-like subset $E$ of $E_\phi$. We choose a subsequence $\{n_i\}_{i \geq 1}$ such that $\liminf\limits_{n \to \infty} \frac{\phi(n)}{n} = \lim\limits_{i \to \infty} \frac{\phi(n_i)}{n_i}$ and $n_{i+1} > n_i + \lfloor\phi(n_i)\rfloor+1$. Denote by $\mathcal{C}_n$ the set of $n$-th order cylinders $I(\varepsilon_1,\cdots,\varepsilon_n)$ satisfying $\varepsilon_k =0$ if $k= n_i+j$ for $i \geq 1$ and $1 \leq j \leq \lfloor\phi(n_i)\rfloor+1$; otherwise, $\varepsilon_k \in \mathcal{A}$. Let
\[
E = \bigcap_{n=1}^\infty \bigcup_{I(\varepsilon_1, \cdots, \varepsilon_n) \in \mathcal{C}_n}I(\varepsilon_1, \cdots, \varepsilon_n).
\]
Then $E \subset E_\phi$. We can also define a measure supported on $E$ like the measure $\nu$ on $F$. Then it can be shown that this measure satisfies the modified mass distributed principle (i.e., Proposition \ref{MMDP}) by similar skills in the proof of Lemma \ref{F}. Thus, we get a lower bound of the Hausdorff dimension of $E_\phi$. Finally, we use the approximation method for the $\beta$-shift in Section \ref{Approximation method} to extend this result from $\beta \in A_0$ to any $\beta > 1$. Combing this with Lemma \ref{E}, we have the following.

\begin{proposition}\label{E quan}
Let $\beta >1$ be a real number.
Assume that $\phi$ is a nonnegative function defined on $\mathbb{N}$.
Then
\begin{equation*}
\dim_{\rm H} E_\phi = \frac{1}{1+\liminf\limits_{n \to \infty} \phi(n)/n}.
\end{equation*}
\end{proposition}

Similar with Proposition \ref{E quan}, replacing $\phi(n)$ by $(\phi(n)-n)$, the following is obtained.

\begin{proposition}\label{E 2}
Let $\beta >1$ be a real number.
Assume that $\phi$ is a nonnegative function defined on $\mathbb{N}$ satisfying $\liminf\limits_{n \to \infty} \phi(n)/n \geq 1$. Then
\begin{equation*}
\dim_{\rm H} \big\{x \in [0,1):\ell_n(x)  \geq \phi(n) -n \ i.o.\big\} = \frac{1}{\liminf\limits_{n \to \infty} \phi(n)/n}.
\end{equation*}
\end{proposition}

\begin{proof}
Let $\liminf\limits_{n \to \infty} \phi(n)/n \geq 1$.
When $\liminf\limits_{n \to \infty} (\phi(n) -n) >0$, then $\phi(n) >n$ holds for sufficiently largr $n$ and hence the result is got via replacing $\phi(n)$ by $(\phi(n)-n)$ from Proposition \ref{E quan}. When $\liminf\limits_{n \to \infty} (\phi(n) -n) <0$, we know that $\phi(n) \leq n$ holds for infinitely many $n \in \mathbb{N}$ and hence $\liminf\limits_{n \to \infty} \phi(n)/n = 1$. Then $\left\{x \in [0,1):\ell_n(x) \geq \phi(n) -n \ i.o.\right\} = [0,1)$ by the definition of $\ell_n(x)$ and hence the desired result is true.
Now let $\liminf\limits_{n \to \infty} (\phi(n) -n) =0$. In this case, if $\limsup\limits_{n \to \infty} (\phi(n) -n) >0$, then $\phi(n) > n$ holds for infinitely many $n \in \mathbb{N}$ and hence the method of the construction of the Cantor-like set in Proposition \ref{E quan} replacing $\phi(n)$ by $(\phi(n)-n)$ is valid; if $\limsup\limits_{n \to \infty} (\phi(n) -n) \leq0$, then $\lim\limits_{n \to \infty} (\phi(n) -n) =0$ and hence $\liminf\limits_{n \to \infty} \phi(n)/n =1$. By the definition of $\ell_n(x)$, we know $\left\{x \in [0,1):\ell_n(x) \geq \phi(n) -n \ i.o.\right\} = [0,1)$ and that the the desired result is obtained.

\end{proof}

During the construction of the Cantor-like set $F$ in Lemma \ref{F}, we can obtain the following lemma by replacing $\phi(n)$ by $(\phi(n)-n)$ and using the approximation method for the $\beta$-shift in Section \ref{Approximation method}.
\begin{lemma}\label{F pian}
Let $\beta >1$ be a real number.
Assume that $\phi$ is a positive and nondecreasing function defined on $\mathbb{N}$ satisfying $\liminf\limits_{n \to \infty} \phi(n)/n \geq 1$. Then
\begin{equation*}
\dim_{\rm H} \left\{x \in [0,1): \limsup_{n \to \infty} \frac{n+\ell_n(x)}{\phi(n)} = 1 \right\} \geq \frac{1}{\liminf\limits_{n \to \infty} \phi(n)/n}.
\end{equation*}
\end{lemma}

\begin{proof}
When $\liminf\limits_{n \to \infty} \phi(n)/n >1$, we obtain that $\phi(n) >n$ holds for sufficiently largr $n$ and that $\liminf\limits_{n \to \infty} (\phi(n)-n) = +\infty$. Note that $\phi$ is nondecreasing,
so both the method of the construction of the Cantor-like set in the proof of Lemma \ref{F} by replacing $\phi(n)$ by $(\phi(n)-n)$ and the approximation method for the $\beta$-shift in Section \ref{Approximation method} are valid. Being similar to the proofs of Lemmas \ref{F} and \ref{F general}, we obtain the desired result.
When $\liminf\limits_{n \to \infty} \phi(n)/n =1$, i.e., $\limsup\limits_{n \to \infty} n/\phi(n) =1$. Since $\frac{n+\ell_n(x)}{\phi(n)} = \frac{n}{\phi(n)}\cdot(1+\frac{\ell_n(x)}{n})$, it is easy to see that
\begin{equation*}
\left\{x \in [0,1): \limsup_{n \to \infty} \frac{n+\ell_n(x)}{\phi(n)}= 1 \right\} \supseteq \left\{x \in [0,1): \lim_{n \to \infty} \frac{\ell_n(x)}{n} = 0 \right\}.
\end{equation*}
Note that
\begin{equation*}
\left\{x \in [0,1): \limsup_{n \to \infty} \frac{\ell_n(x)}{n} = 0 \right\} \supseteq \left\{x \in [0,1): \limsup_{n \to \infty} \frac{\ell_n(x)}{\log n} = 1 \right\},
\end{equation*}
by Lemma \ref{F general}, we obtain that
\begin{equation*}
\dim_{\rm H} \left\{x \in [0,1): \lim_{n \to \infty} \frac{n+\ell_n(x)}{\phi(n)} = 1 \right\} \geq 1.
\end{equation*}
\end{proof}

In view of Lemmas \ref{E} and \ref{F pian}, we have

\begin{proposition}\label{F pian2}
Let $\beta >1$ be a real number.
Assume that $\phi$ is a positive and nondecreasing function defined on $\mathbb{N}$ satisfying $\liminf\limits_{n \to \infty} \phi(n)/n \geq 1$. Then
\begin{equation*}
\dim_{\rm H}  \left\{x \in [0,1): \limsup_{n \to \infty} \frac{n+\ell_n(x)}{\phi(n)} = 1 \right\} = \frac{1}{\liminf\limits_{n \to \infty} \phi(n)/n}.
\end{equation*}
\end{proposition}

\section{The proof of Theorem \ref{inf dimension}}

In this section, we will prove Theorem \ref{inf dimension}.

\begin{proof}[Proof of Theorem \ref{inf dimension}]

(i) If $\eta =0$, we have $\limsup\limits_{n \to \infty} n /\phi(n) = +\infty$. By the definition of $\ell_n(x)$ defined in Section 3,
we deduce that
\[
\limsup_{n \to \infty} \frac{n + l_n(x)}{\phi(n)} \geq \limsup_{n \to \infty} \frac{n }{\phi(n)} = +\infty
\]
for any $x \in [0,1)$. It follows from the second inequality of (\ref{jie}) that
\begin{equation}\label{jie 2}
\limsup_{n \to \infty} \frac{n + \ell_n(x)}{\phi(n)} \leq - \liminf_{n \to \infty} \frac{1}{\phi(n)}\log_\beta (x -\omega_n(x))
\end{equation}
for any $x \in [0,1)$. Thus,
\[
\liminf_{n \to \infty} \frac{1}{\phi(n)}\log_\beta (x -\omega_n(x)) = - \infty
\]
for any $x \in [0,1)$. If $0<\eta <1$, by the definition of $\ell_n(x)$, we obtain that
\[
\limsup_{n \to \infty} \frac{n + \ell_n(x)}{\phi(n)} \geq \limsup_{n \to \infty} \frac{n}{\phi(n)} = \frac{1}{\liminf\limits_{n \to \infty} \phi(n)/n}= \frac{1}{\eta} > 1
\]
for any $x \in [0,1)$. In view of (\ref{jie 2}), we deduce that
\[
\liminf_{n \to \infty} \frac{1}{\phi(n)}\log_\beta (x -\omega_n(x)) < - 1\ \ \text{for any $x \in [0,1)$}.
\]
Therefore, the set
\begin{equation*}
B_\phi = \left\{x \in [0,1): \liminf_{n \to \infty} \frac{1}{\phi(n)}\log_\beta (x -\omega_n(x)) = -1 \right\}
\end{equation*}
is empty.

(ii) Let $\eta \geq 1$. So $\phi(n) \to \infty$ as $n \to \infty$ and then it follows from (\ref{jie}) that
\begin{equation*}
\limsup_{n \to \infty} \frac{n + \ell_n(x)}{\phi(n)} =- \liminf_{n \to \infty} \frac{1}{\phi(n)}\log_\beta (x -\omega_n(x))
\end{equation*}
for any $x \in [0,1)$. Note that $\phi$ is nondecreasing, by Proposition \ref{F pian2}, we obtain  $\dim_{\rm H} A_\phi = 1/\eta$.
\end{proof}

\section{Applications}

\subsection{The orbits of real numbers under $\beta$-transformation}

For any $x \in [0,1)$, we say that the sequence $x,T_\beta x, \cdots, T_\beta^nx, \cdots$ is \emph{the orbit of $x$ under $T_\beta$}.
In 2011, Li and Chen \cite{lesL.C11} studied the topological properties of the orbits of $x \in [0,1)$ under $T_\beta$ by considering the set
\begin{equation*}
O_\beta = \left\{x \in [0,1): \text{the orbit of}\ x\ \text{under}\ T_\beta\ \text{is dense in} \ [0,1] \right\}.
\end{equation*}
They proved that $O_\beta$ is uncountable, dense, of the second category and has full Lebesgue measure, while its complementary set $O_\beta^c$  is uncountable, dense, of the first category and has full Hausdorff dimension. They also showed that the $\beta$-transformation $T_\beta$ is chaotic in the sense of both Li-Yorke and Devaney. Recently, Ban and Li \cite{lesBL14} studied the multifractal spectra for the recurrence rate of the first return time of $\beta$-transformation, including the cases returning to the ball and cylinder.

For any $n \in \mathbb{N}$, by the definitions of $\omega_n(x)$ and $T_\beta$ in Section 1, we obtain
\begin{equation}\label{orbit}
x - \omega_n(x) = \frac{T_\beta^nx}{\beta^n}.
\end{equation}
Combing this with (\ref{jie}), we have
\begin{equation}\label{Orbit}
\frac{1}{\beta^{\ell_n(x)+1}} \leq  T_\beta^nx \leq \frac{1}{\beta^{\ell_n(x)}}.
\end{equation}
As a consequence of Proposition \ref{sup ln}, we give a quantitative formula for the growth speed of the orbit of $x $ under $T_\beta$.

\begin{theorem}\label{orbit theorem}
Let $\beta >1$ be a real number. Then for $\lambda$-almost all $x \in [0,1)$,
\[
\liminf_{n \to \infty} \frac{\log T_\beta^nx}{\log n} = -1.
\]
Moreover, for any real number $x \in [0,1)$ whose $\beta$-expansion is infinite, we have
\[
\limsup_{n \to \infty} \frac{\log T_\beta^nx}{\log n} = 0.
\]
\end{theorem}

By Proposition \ref{F quan} and the inequalities (\ref{Orbit}), we obtain the Hausdorff dimension of exceptional set according to the above metric result.

\begin{theorem}\label{th1 ex}
Let $\beta >1$ be a real number and $\phi$ be a positive and nondecreasing function defined on $\mathbb{N}$ with $\phi(n) \to \infty$ as $n \to \infty$. Denote $\eta:= \liminf\limits_{n \to \infty} \phi(n)/n$. Then
\begin{equation*}
\dim_{\rm H} \left\{x \in [0,1): \liminf_{n \to \infty} \frac{\log T_\beta^nx}{\phi(n)} = -1 \right\} = \frac{1}{1+ \eta/\log \beta}.
\end{equation*}
\end{theorem}

As an direct application of Theorem \ref{th1 ex}, we obtain the Hausdorff dimensions of the level sets.

\begin{corollary}\label{6}
Let $\beta >1$ be a real number. Then for any $\alpha \geq 0$,
\begin{equation*}
\dim_{\rm H} \left\{x \in [0,1): \liminf_{n \to \infty} \frac{\log T_\beta^nx}{\log n} = -\alpha \right\} = 1.
\end{equation*}
\end{corollary}

\begin{proof}
For $\alpha > 0$, taking $\phi(n) = \alpha \log n $, then $\phi(n)$ is positive and nondecreasing with $\phi(n) \to \infty$ as $n \to \infty$. Applying these $\phi(n)$ to Theorem \ref{th1 ex}, note that $\liminf\limits_{n \to \infty} \phi(n)/n =0$, we obtain the desired result. Now let $\alpha = 0$. Notice that
\begin{equation*}
\left\{x \in [0,1): \liminf_{n \to \infty} \frac{\log T_\beta^nx}{\log n} =0 \right\} \supseteq
\left\{x \in [0,1): \liminf_{n \to \infty} \frac{\log T_\beta^nx}{\log\log n} = -1 \right\}
\end{equation*}
and the later set has full Hausdorff dimension by Theorem \ref{th1 ex}, so we complete the proof.
\end{proof}

Application of Corollary \ref{6} indicates that the set of points such that the metric result in Theorem \ref{orbit theorem} does not hold has full Hausdorff dimension.

\subsection{The shrinking target type problem for $\beta$-transformations}

The typical shrinking target problem with the given target aims at investigating the Hausdorff dimensions of sets of points whose orbits are close to some previously chosen point (see Hill and Velani \cite{lesHV95}). In fact, Bugeaud and Wang \cite{lesB.W14} has given a more general result about this problem for $\beta$-expansions. Tan and Wang \cite{lesTW11} concerned the quantitative recurrence properties of the beta dynamical system $([0,1],T_\beta)$ for general $\beta > 1$ and the Hausdorff dimension of the set of points with the prescribed recurrence rate is determined. Liao and Seuret \cite{lesLS13} studied the
shrinking target problem for expanding Markov maps with a finite partition.
Recently, Seuret and Wang \cite{lesSW15} investigated a quantitative version of Poincar\'{e}'s recurrence theorem in a conformal iterated function system, which includes the dynamical systems of $p$-adic expansions, continued fraction expansion, as well as some dynamical systems defined on fractal sets. For more applications about the shrinking target problem, see \cite{lesKim07, lesL.P.W.W14, lesLWWX14, lesLW16}.

By Proposition \ref{E quan} and the inequalities (\ref{Orbit}), we have the following Shrinking target problem for $\beta$-transformations which is a special case of Bugeaud and Wang \cite[Theorem 1.6]{lesB.W14}.

\begin{theorem}[Bugeaud and Wang \cite{lesB.W14}]
Let $\beta>1$ be a real number and $\phi$ be a positive function defined on $\mathbb{N}$. Then
\[
\dim_{\rm H} \left\{x \in [0,1): T_\beta^nx  \leq \beta^{-\phi(n)} \ i.o.\right\}= \frac{1}{1+\liminf\limits_{n \to \infty} \phi(n)/n}.
\]
\end{theorem}

\begin{proof}
It follows from (\ref{Orbit}) that
\[
\left\{x \in [0,1): \ell_n(x) \geq \phi(n) \ i.o.\right\} \subseteq \left\{x \in [0,1): T_\beta^nx  \leq \beta^{-\phi(n)} \ i.o.\right\} \subseteq \left\{x \in [0,1): \ell_n(x) \geq \phi(n)-1 \ i.o.\right\}.
\]
In view of Proposition \ref{E quan}, we know both left-hand and right-hand sets are of Hausdorff dimension $1/(1+\liminf\limits_{n \to \infty} \phi(n)/n)$ and hence we complete the proof.
\end{proof}

\subsection{The Diophantine-type problem for $\beta$-expansions}

Notice that the classical Diophantine questions concern the dimension of the set
\[
\left\{x \in [0,1): \left|x- p/q\right| \leq  q^{-2\tau}\ \text{for infinitely many couples } (p,q) =1 \right\},
\]
the following is interpreted as a Diophantine approximation problem of $\beta$-expansions.
In fact, in view of (\ref{orbit}), this problem can also be viewed as a shrinking target problem for $\beta$-transformations.

\begin{theorem}[Bugeaud and Wang \cite{lesB.W14}]
Let $\beta >1$ be a real number and $\phi$ be a positive function defined on $\mathbb{N}$.
Denote $\eta = \liminf\limits_{n \to \infty} \phi(n)/n $ and
\begin{equation*}
C_{\phi}=\left\{x \in [0,1): x -\omega_n(x) \leq  \beta^{-\phi(n)} \ i.o.\right\}
\end{equation*}
Then\\
(i) If $0 \leq \eta < 1$, then $\dim_{\rm H} B_\phi =1$.\\
(ii) If $\eta \geq 1$, then $\dim_{\rm H} B_\phi =1/\eta$.
\end{theorem}

\begin{proof}
In view of (\ref{jie}), we deduce that
\[
\left\{x \in [0,1):\ell_n(x)  \geq \phi(n) -n\ i.o.\right\} \subseteq C_\phi \subseteq \left\{x \in [0,1):\ell_n(x)  \geq \phi(n) -n-1 \ i.o.\right\}.
\]
When $0 \leq \eta < 1$, we obtain $\phi(n) -n \leq 0$ for infinitely many $n \in \mathbb{N}$. Thus, by the definition of $\ell_n(x)$, the set of points $x$ such that $\ell_n(x)  \geq \phi(n) -n$ for infinitely many $n$ is the unit interval $[0,1)$
and hence that $\dim_{\rm H} B_\phi =1$. When $\eta \geq 1$, it follows from Proposition \ref{E 2} that $\dim_{\rm H} B_\phi =1/\eta$.
\end{proof}


\subsection{The run-length function}
Recall that
\begin{equation*}
r_n(x) = \sup\big\{k \geq 0: \varepsilon_{i+1}(x)= \cdots = \varepsilon_{i+k}(x) =0\ \text{for some}\  0 \leq i \leq n-k\big\},
\end{equation*}
i.e., the maximal length of consecutive zeros in the first $n$ digits of the $\beta$-expansion of $x$. Let
\[
G_\phi = \left\{x \in [0,1): \limsup_{n \to \infty} \frac{r_n(x)}{\phi(n)} = 1 \right\},
\]
where $\phi$ is a positive function defined on $\mathbb{N}$. Now let $\mathcal{H}$ denote the set of positive and nondecreasing functions defined on $\mathbb{N}$ satisfying $\phi(n) \to \infty$ as $n \to \infty$ and
\begin{equation}\label{condition}
 \lim_{n \to \infty} \frac{\phi(n +\phi(n))}{\phi(n)} =1.
\end{equation}
It is noting that the following examples of $\phi$ are the typical elements of $\mathcal{H}$.
\begin{itemize}
  \item $\phi(n) = \alpha n^\gamma$ with $\alpha >0$ and $0<\gamma<1$.
  \item $\phi(n) =  \alpha(\log n)^\gamma$ with $\alpha >0$ and $\gamma>0$.
  \item $\phi(n) = \alpha n/(\log n)^\gamma$ with $\alpha >0$ and $\gamma>0$.
\end{itemize}

The next two lemmas are the basic properties of the function in $\mathcal{H}$.

\begin{lemma}\label{condition equivalent}
Let $\phi \in \mathcal{H}$. Then
\[
\lim_{n \to \infty} \frac{\phi(n +\delta\phi(n))}{\phi(n)} =1
\]
for any integer $\delta > 0$.
\end{lemma}

\begin{proof}
Let $\phi \in \mathcal{H}$. Then limit (\ref{condition}) holds. For any $\varepsilon >0$, there exists $N:=N_\varepsilon >0$ such that for all $n \geq N_\varepsilon$, we have
\begin{equation}\label{limphi}
\phi(n + \phi(n)) \leq (1+\varepsilon)\phi(n).
\end{equation}
For any integer $j > 0$, by the monotonic property of $\phi$, we have $\phi(n + \phi(n)) \geq \phi(n)$ and hence that
\[
\phi\big(n + j\phi(n) + \phi(n + j\phi(n))\big) \geq \phi\big(n + (j+1)\phi(n)\big).
\]
On the other hand, in view of (\ref{limphi}), we obtain that
\[
\phi\big(n + j\phi(n) + \phi(n + j\phi(n))\big) \leq (1+\varepsilon)\phi\big(n + j\phi(n)\big)
\]
by regarding $(n + j\phi(n))$ as a whole. Therefore, for any integer $j$ and $n \geq N$, we have
\[
\phi\big(n + (j+1)\phi(n)\big) \leq (1+\varepsilon)\phi\big(n + j\phi(n)\big).
\]
For any integer $\delta >0$ and $n \geq N$, we obtain that
\[
1 \leq \frac{\phi(n +\delta\phi(n))}{\phi(n)} = \prod_{j=0}^{\delta-1} \frac{\phi\big(n + (j+1)\phi(n)\big)}{\phi\big(n + j\phi(n)\big)} \leq (1+\varepsilon)^\delta
\]
holds for any $\varepsilon >0$, which implies the desired result.
\end{proof}

\begin{lemma}\label{infequivalent}
Let $\phi \in \mathcal{H}$. Then
$\liminf\limits_{n \to \infty} \phi(n)/n=0$.
\end{lemma}

\begin{proof}
We assume that $\liminf\limits_{n \to \infty} \phi(n)/n= \eta >0$. Let $\{n_k\}_{k \geq 1}$ be a subsequence such that $\eta = \lim\limits_{k \to \infty} \phi(n_k)/n_k$. For any $\varepsilon>0$, there exists $K£º=K_\varepsilon >0$ such that for all $k \geq K$, we have $\phi(n_k)\geq (\eta -\varepsilon)n_k$.
By the monotonic property of $\phi$, we deduce that
\[
\phi\big(n_k+\phi(n_k)\big) \geq \phi\big(n_k+(\eta-\varepsilon)n_k\big) \geq (\eta-\varepsilon)(1+\eta-\varepsilon)n_k
\]
for any $k \geq K$. Therefore, $\phi\big(n_k+\phi(n_k)\big)/\phi(n_k) \geq (\eta-\varepsilon)(1+\eta-\varepsilon)n_k/\phi(n_k)$ for any $k \geq K$ and hence that
\[
\limsup_{k \to \infty} \frac{\phi\big(n_k+\phi(n_k)\big)}{\phi(n_k)} \geq \limsup_{k \to \infty} \frac{(\eta-\varepsilon)(1+\eta-\varepsilon)n_k}{\phi(n_k)} =
(1+\eta-\varepsilon)\left(1-\frac{\varepsilon}{\eta}\right).
\]
Since $\varepsilon >0$ is arbitrary, we have
\[
\limsup_{n \to \infty} \frac{\phi(n +\phi(n))}{\phi(n)} \geq 1+\eta >1,
\]
which contradicts the condition $\phi \in \mathcal{H}$.
\end{proof}

By the relation between $r_n(x)$ and $\ell_n(x)$, we state that

\begin{lemma}\label{rn subset}
Let $\phi \in \mathcal{H}$. Then
\[
\left\{x \in [0,1): \limsup_{n \to \infty} \frac{\ell_n(x)}{\phi(n)} = 1 \right\} \subseteq \left\{x \in [0,1): \limsup_{n \to \infty} \frac{r_n(x)}{\phi(n)} = 1 \right\}.
\]
\end{lemma}

\begin{proof}
On the one hand, for any $x \in [0,1)$, in view of the definitions of $r_n(x)$ and $\ell_n(x)$, there exists $1 \leq k_n := k_n(x) < n$ such that $r_{n}(x) = \ell_{k_n}(x)$ and hence
\[
\frac{r_{n}(x)}{\phi(n)} =\frac{\ell_{k_n}(x)}{\phi(n)} \leq \frac{\ell_{k_n}(x)}{\phi(k_n)}
\]
since $\phi$ is positive and nondecreasing. Therefore,
\begin{equation}\label{rn1}
\limsup_{n \to \infty}\frac{r_{n}(x)}{\phi(n)} \leq
\limsup_{n \to \infty}\frac{\ell_{k_n}(x)}{\phi(k_n)} \leq
\limsup_{n \to \infty}\frac{\ell_{n}(x)}{\phi(n)}.
\end{equation}
We suppose that $\phi \in \mathcal{H}$. On the other hand, for any $x \in [0,1)$,
note that $r_{n+\ell_n(x)}(x) = \max_{1 \leq k \leq n}\ell_k(x)$, then $r_{n+\ell_n(x)}(x) \geq \ell_{n}(x)$ and hence
\begin{equation}\label{rn2}
\limsup_{n \to \infty} \frac{r_n(x)}{\phi(n)} \geq \limsup_{n \to \infty} \frac{r_{n+\ell_n(x)}(x)}{\phi(n+\ell_n(x))} \geq \limsup_{n \to \infty} \left(\frac{\phi(n)}{\phi(n+\ell_n(x))} \cdot  \frac{\ell_n(x)}{\phi(n)}\right).
\end{equation}
Now let $x$ be a real number satisfying $\limsup\limits_{n \to \infty} \ell_n(x) /\phi(n) = 1$. Then there exists $N:=N(x) >0$ such that $0 \leq \ell_n(x) \leq 2 \phi(n)$ for all $n \geq N$ and hence $\phi(n) \leq \phi(n + \ell_n(x)) \leq \phi(n + 2\phi(n))$ since $\phi$ is nondecreasing. By Lemma \ref{condition equivalent}, we eventually deduce that
\[
\lim_{n \to \infty}  \frac{\phi(n)}{\phi(n+\ell_n(x))} =1.
\]
In view of (\ref{rn2}), we obtain that
\[
\limsup_{n \to \infty} \frac{r_n(x)}{\phi(n)} \geq \limsup_{n \to \infty} \left(\frac{\phi(n)}{\phi(n+\ell_n(x))} \cdot  \frac{\ell_n(x)}{\phi(n)}\right) = \limsup_{n \to \infty} \frac{\ell_n(x)}{\phi(n)} =1.
\]
Combing this with (\ref{rn1}), we complete the proof.
\end{proof}

\begin{theorem}
Let $\phi \in \mathcal{H}$. Then
$\dim_{\rm H} G_\phi =1$.
\end{theorem}

\begin{proof}
In view of Lemma \ref{rn subset}, we know that $F_\phi \subseteq G_\phi$ and hence it follows from Proposition \ref{F quan} that
\[
\dim_{\rm H} G_\phi \geq \frac{1}{1+\liminf\limits_{n \to \infty} \phi(n)/n}.
\]
Let $\phi \in \mathcal{H}$. By Lemma \ref{infequivalent}, we have $\liminf\limits_{n \to \infty} \phi(n)/n =0$. Therefore, we get $\dim_{\rm H} G_\phi =1$.
\end{proof}

\begin{corollary}
For any $\alpha \geq 0$,
\[
\dim_{\rm H} \left\{x \in [0,1): \limsup_{n \to \infty} \frac{r_n(x)}{\log_{\beta} n} = \alpha \right\} =1.
\]
\end{corollary}

{\bf Acknowledgement:}
The work was supported by NSFC 11371148, 11201155 and Guangdong Natural Science Foundation 2014A030313230. The authors thank Professor Christoph Bandt for the comments on the results.

\end{document}